\documentclass[a4paper,11pt, reqno]{amsart}
\usepackage[latin1]{inputenc}

\usepackage{amssymb}
\usepackage{amsmath}
\usepackage{amsthm}
\usepackage{graphicx}
\usepackage{a4wide}
\usepackage{geometry}
\usepackage{amsmath}
\usepackage{amsfonts}
\usepackage{amssymb}
\usepackage{graphicx}%
\providecommand{\U}[1]{\protect\rule{.1in}{.1in}}
\renewcommand{\ll}{{\ell}}

\newcommand{\be}[1]{\begin{equation}\label{#1}}
\newcommand{\ee}{\end{equation}}

\let\pa\partial
\let\r\rho

\newtheorem{thm}{Theorem}[section]

\newtheorem{df}{Definition}[section]
\newtheorem{prop}{Proposition}[section]

\newtheorem{rem}{Remark}[section]
\newtheorem{cor}{Corollary}[section]
\numberwithin{equation}{section}

\newcommand{\beq}{\begin{eqnarray}}
\newcommand{\eeq}{\end{eqnarray}}
\newcommand{\beqs}{\begin{eqnarray*}}
\newcommand{\eeqs}{\end{eqnarray*}}
\newcommand{\bequ}{\begin{equation}}
\newcommand{\eequ}{\end{equation}}

\def\r{\rho}

\def\T{{\cal T}}
\def\M{{\cal M}}
\newcommand{\cal}{\mathcal}

\def\e{\varepsilon}

\newtheorem{lemma}{Lemma}[section]

\let\pa\partial

\makeatother

\date{October 1, 2016}
\begin{document}
\title[Global well-posedness nonlinear moisture dynamics with phase changes]{Global well-posedness for passively transported nonlinear moisture dynamics with phase changes}

\author{Sabine Hittmeir}\address[Sabine Hittmeir]{Fakult\"at f\"ur Mathematik, Universit\"at Wien, Oskar-Morgenstern-Platz 1,
1090 Wien, Austria}
\email{sabine.hittmeir@univie.ac.at}

\author{Rupert Klein}\address[Rupert Klein]{
FB Mathematik \& Informatik, Freie Universit\"at Berlin,
Arnimallee 6, 14195 Berlin, Germany}
\email{rupert.klein@math.fu-berlin.de}

\author{Jinkai Li}\address[Jinkai Li]{   Department of Mathematics, The Chinese University of Hong Kong, Shatin, N.T., Hong Kong. Also:
     Department of Computer Science and Applied Mathematics, Weizmann Institute of Science, Rehovot 76100, Israel}
\email{jklimath@gmail.com}

\author{Edriss~S.~Titi}
\address[Edriss~S.~Titi]{
Department of Mathematics, Texas A\&M University, 3368 TAMU, College Station, TX 77843-3368, USA. ALSO, Department of Computer Science and Applied Mathematics, Weizmann Institute of Science, Rehovot 76100, Israel.}
\email{titi@math.tamu.edu and edriss.titi@weizmann.ac.il}

\maketitle
\begin{abstract}
We study a moisture model for warm clouds that has been used by Klein and Majda in \cite{KM} as a basis for multiscale asymptotic expansions for deep convective phenomena. These moisture balance equations correspond to a  bulk microphysics closure in the spirit of Kessler \cite{Ke} and Grabowski and Smolarkiewicz \cite{GS}, in which water is present in the gaseous state as water vapor and in the liquid phase as cloud water and rain water. It thereby contains closures for the phase changes condensation and evaporation, as well as the processes of autoconversion of cloud water into rainwater and the collection of cloud water by the falling rain droplets.
Phase changes are associated with enormous amounts of latent heat and therefore provide a strong coupling to the thermodynamic equation.

In this work we assume the velocity field to be given and prove rigorously the global existence and uniqueness of uniformly bounded solutions
of the moisture model with viscosity, diffusion and heat conduction. To guarantee local well-posedness we first need to establish local existence results for linear parabolic equations, subject to the Robin boundary conditions on the cylindric type of domains under consideration.
 We then derive a priori estimates, for proving the maximum principle, using the Stampacchia method, as well as the iterative method by Alikakos \cite{A} to obtain uniform boundedness.
The evaporation term is of power law type, with an exponent in general less or equal to one and therefore making the proof of uniqueness more challenging. However, these difficulties can be  circumvented by introducing new unknowns, which satisfy the required cancellation and monotonicity properties in
the source terms.

\end{abstract}

\allowdisplaybreaks
\section{Introduction}
Latent heat conversions due to phase changes of water in the atmosphere do
to a great extent influence the energy balance, which is made evident by
the following statement of Emanuel in \cite{S} on p.\,5: \emph{``If all
the water vapor near the surface on a hot muggy day were condensed out,
the air would warm by about $35^\circ C$"}. Although the total amount of
moisture even under saturated conditions is small compared to the dry
air components, its effect on the dynamics can be enormous, getting
visible e.g.\,in thunderstorms. To obtain a better understanding of such
complex processes involving the interaction of deep convective phenomena
with their environments,
Klein and Majda \cite{KM}  incorporated moisture processes into the
asymptotic framework by performing very careful nondimensionalisations.
Their studies involving multiple scales are thereby based on a warm cloud
model, where water is present in gaseous and liquid form. Using this setting
in further works they in particular also revealed interesting new results on the
modulation of gravity waves by columnar clouds \cite{RK,RKM}. The
closure for the source terms of the moisture quantities in that work corresponds
to a basic form of a bulk microphysics model in the spirit of
Kessler \cite{Ke} and Grabowski and Smolarkiewicz \cite{GS}. In the present
work we establish global well-posedness of the moisture model coupled through
the phase changes to the thermodynamic equation, where we assume the velocity
field to be given. The coupling to the primitive equations will be studied in
a forthcoming paper by taking over  the ideas of Cao and Titi \cite{CT} for
their recent breakthrough on the global solvability of the latter system.
Moreover, cases of  partial diffusions, arising from the asymptotical analysis in \cite{KM}, will be analyzed in a future work too based on the results by Cao et al.\,\cite{CLT1,CLT2,CLT3,CLT4}.

To the best of our knowledge the mathematical analysis of
moisture models with phase changes has been investigated only
in a few papers, see Coti-Zelati et al.\,\cite{BCZ,CZF,CZH,CZT}, Li and Titi \cite{LITITI-TCM2}, and Majda and Souganidis \cite{MAJSOU}.
The models studied by Coti-Zelati et al.\,\cite{BCZ,CZF,CZH,CZT} consist
of one moisture quantity coupled to the temperature and contain only the process
of condensation
during upward motion, see e.g.\,\cite{HW}. Since the source term  considered in \cite{BCZ,CZF,CZH,CZT} is
modeled via a Heavy side function as a switching term between
saturated and undersaturated regions, an approach based on differential
inclusions and variational techniques was used in these papers.
The moisture model investigated in \cite{LITITI-TCM2,MAJSOU} is a coupled nonlinear system of the moisture to the barotropic and the first baroclinic models of the velocity, where the source term in the moisture
is the precipitation.  This  model proposed by Frierson et al.\,\cite{FMP} for the tropical atmosphere
involves a small convective adjustment time
scale parameter  in the precipitation term.  The rigorous analysis was carried out in \cite{LITITI-TCM2} for both the finite-time
relaxation system and the instantaneous relaxation limiting system,
including the global well-posedness and the strong convergence
of the relaxation limit. Some moisture models without phase changes coupled to the primitive equations were also considered by Guo and Huang \cite{GH1,GH2}.

The model we are analyzing in this paper is physically more refined and consists of three moisture quantities for  water vapor, cloud water and rain water and contains besides all the phase changes due to condensation and evaporation also the autoconversion of cloud water to rain water after a certain threshold is reached, as well as a closure for the collection of cloud water by the falling rain droplets.

In the remainder of the introduction we first introduce the
moisture model in the setting of Klein and Majda \cite{KM} with additional diffusion, viscosity and heat conduction in
cartesian coordinates. We then reformulate the system in pressure
coordinates, which have the advantage that under the assumption
of hydrostatic balance the continuity equation takes the form of
the incompressibility condition. Although we assume the velocity field to be given this property is very useful when deriving a priori estimates, since it allows for cancellations of integral terms involving the advection terms. In the pressure coordinates the vertical diffusion term becomes nonlinear and we  make here the standard approximation by linearizing around a given background temperature profile, see also \cite{BCZ, CZF, LTW, PTZ}. In section \ref{sec.main} we
then formulate the full problem with boundary and side conditions
and state the main result on the global existence and uniqueness
of solutions.

Section \ref{sec.comp} contains the proof of local existence and some appropriate a priori estimates. The local existence is proved based on the
results established in section \ref{secellipticparabolic} for linear
parabolic equations, subject to the Robin boundary conditions on the cylindric type of domains.
The a priori estimates are derived by using the Stampaccia type arguments
and the iterative methods of Alikakos \cite{A}.

Section \ref{sec.uni} contains the proof of the well-posedness, i.e.,  the uniqueness and the continuous dependence of the solutions on the initial data. The proof is based on the typical $L^2$-type estimates for the difference between two solutions. Here the main difficulty is caused by the evaporation source term which is of power law type in $q_r$ with an exponent, generally, less or equal to $1$, see (\ref{Sev}), below.
This possible lack of Lipschitz continuity is overcome by introducing new unknowns, for which the source terms satisfy certain cancellation and monotonicity properties  in the estimates, such that the uniqueness of the solution can be concluded also for the original unknowns.

The last section, section \ref{secellipticparabolic}, is a technical section,
which gives the necessary results for some linear parabolic equations subject
to the Robin boundary conditions on the cylindric domains. These parabolic
equations are motivated by the problem considered in the present paper.
Due to the cylindric domains not being smooth, the corresponding
parabolic theory required for our analysis cannot be found explicitly in the existing literature. Therefore, for the sake of completeness, the derivations are
carried out here in detail.

\subsection{Moisture model in cartesian coordinates}\
In the following $\r, T, p, {\bf{u}}=(u,v), w$ are the density,  temperature, pressure, as well as horizontal and vertical velocity components respectively. The thermodynamic equation is given by
\beq
\frac{DT}{Dt} -\frac{RT}{c_p p }\frac{Dp}{Dt}=  S_T+ {\cal D}^TT\,,
\eeq
where ${\cal D}$ denotes here, and in what follows,   the turbulent or molecular  diffusion/viscosity  operators,  for which we will assume the simple closure
in form of the  full Laplacian. The term $S_T$ accounts for the diabatic source and sink terms, such as latent heating, radiation effects, etc., (see e.g.\,also \cite{CZF,CZH,KM}), but we will in the following only focus on the effect of latent heat in association with phase changes.
The total derivative contains the advection with respect to the given velocity field $(u,v,w)$ is denoted by:
\beq
\frac{D}{Dt}=\pa_t + u \pa_x+v\pa_y+w\pa_z\,.
\eeq
The ideal gas law with the gas constant $R$ reads
\beq\label{ideal}
p=R\r T\,.
\eeq
\begin{rem}
 Due to the difference of the gas constant for dry air and water vapor, for moist air more precisely the virtual temperature, or the density temperature
respectively if also condensed water is present, should be used in \eqref{ideal}. Moreover, the heat capacities in the presence of moisture vary as well. Notably, since these corrections are small, they are neglected in the following. For more details on the thermodynamics of moist air we refer, e.g., to \cite{E}.
\end{rem}
 To describe the state of the atmosphere a common thermodynamic quantity used is the potential temperature
\begin{equation}\label{PT}
\theta = T \left(\frac{p_0}{p}\right)^\frac{\gamma -1}{\gamma}\,,
\end{equation}
which is conserved during isentropic motions. Here and in the following the isentropic exponent $\gamma$ denotes the ratio of heat capacities $\gamma=c_p/c_V$ satisfying
\[\frac{\gamma -1}{\gamma}=\frac{R}{c_p}\,.\]
Then $\theta$ satisfies
\beq
\frac{D\theta}{Dt}  =  \frac{\theta}{T}(S_T + {\cal D}^T T)\,.
\eeq
In the case of moisture being present typically  the water vapor mixing ratio, defined as the ratio of the density of $\r_v$ over the density
of dry air $\r_d$\,,
\beq
q_{v}=\frac{\r_v}{\r_d}\,,
\eeq
is used for a measure of quantification.
If saturation effects occur, then water is also present in liquid form as cloud water and rain water which are represented by the additional moisture quantities
\beq
 q_{c}=\frac{\r_c}{\r_d}\,,\qquad q_{r}=\frac{\r_r}{\r_d}\,.
 \eeq
We focus here on warm clouds, where water is present only in gaseous and liquid form, i.e., no ice or snow phases occur.

For these mixing ratios for water vapor, cloud water and rain water we have the following moisture balances
\begin{eqnarray}
 \frac{D q_v}{Dt} &=& S_{ev} - S_{cd}+\cal{D}^{q_v} q_v\,,\\
\frac{D q_{c}}{Dt} &=& S_{cd} - S_{ac}- S_{cr}+\cal{D}^{q_c} q_c\,,\\
\frac{D q_{r}}{Dt} -\frac{V}{g\r}\pa_z(\r q_r) &=& S_{ac} + S_{cr}- S_{ev}+\cal{D}^{q_r} q_r\,,
\end{eqnarray}
where  $S_{ev}, S_{cd}, S_{ac}, S_{cr}$ are the rates of evaporation of rain water, the condensation of water vapor to cloud water and
the inverse evaporation process, the auto-conversion of cloud water into rainwater by accumulation of microscopic droplets,
and the collection of cloud water by falling rain. Moreover $V$ denotes the terminal velocity of falling rain and is assumed to be constant.

The key quantity to appear in the following explicit expressions for the source terms is the saturation mixing ratio
\beq q_{vs}=\frac{\r_{vs}}{\r_d}\,,\eeq
which gives the threshold for saturation as follows
\begin{itemize}
\item undersaturated region: $q_v<q_{vs}$,
\item saturated region: $q_v = q_{vs}$,
\item oversaturated region: $q_v>q_{vs}$.
\end{itemize}
Denoting $E=R/R_v$ as the ratio of the individual gas constants for dry air and water vapor, the saturation vapor mixing ratio satisfies
\beq
q_{vs}(p,T)=\frac{E e_s(T)}{p-e_s(T)},
\eeq
with the saturation vapor pressure $e_s$ as a function of $T$ being defined by the Clausius-Clapeyron equation:
\beq
e_s(T)=e_s(T_0)\exp\left(\frac{L}{R_v}\left(\frac{1}{T_0}-\frac{1}{T}\right)\right)\,,
\eeq
where the latent heat per unit mass of water vapor $L$ is here and in the following assumed to be constant, which in general
varies slightly with temperature.
From this formula it is obvious that $e_s$ increases in $T$, thereby quantifying the fact that the warmer the air is, the more moisture it can carry.
Only positive temperature values are meaningful.
Typically, the reference temperature $T_0=273.15K$ is used.

The Clausius-Clapeyron equation is only meaningful
for temperature ranges found in the troposphere, thus in particular
as $e_s$ vanishes with decreasing values in $T$, we shall pose in the following the natural assumption
\beq
\label{cond.qvs}
\label{nonneg.qvs}
e_s(T)=0 \, \quad \textnormal{and} \quad q_{vs}(p,T)=0, \quad \textnormal{for} \  T \leq \underline{T}\,,
\eeq
for some constant $\underline{T}\geq 0\, \textit{K}$, which will also
be helpful for proving nonnegativity of the solutions.
Moreover the saturation mixing ratio $q_{vs}$ is assumed to be nonnegative
and bounded, that is
\begin{equation}
  \label{1.3-1}
  0\leq q_{vs}(p,T)\leq q_{vs}^*,
\end{equation}
for a positive constant $q_{vs}^*$.
For deriving uniqueness of the solutions we need additionally the Lipschitz continuity of $q_{vs}$, i.e., we assume
\beq\label{LC}|q_{vs}(p, T_1)-q_{vs}(p,T_2)|\leq C|T_1-T_2|\,,\eeq
where $C$ is a positive constant.

\subsection{Explicit expressions for the source terms}

Condensation sets free enormous amounts of latent heat, which causes a warming of the surrounding environment,
and therefore is a source in the thermodynamic equation.
Evaporation on the other hand enters as a sink displaying its cooling effect.
The temperature source term accounting for the effects due to latent heat is therefore given by
\beq
S_T=\frac{L}{c_p}(S_{cd}-S_{ev})\,,
\eeq
see e.g.\,\cite{E,KM}.
All other non-moisture related diabatic sources for potential temperature like e.g.\,radiation effects are not being considered here.
For the source terms of the mixing ratios we take over the setting of Klein and Majda \cite{KM} corresponding to a  bulk microphysics closure in the spirit of Kessler \cite{K} and Grabowski and
Smolarkiewicz \cite{GS}:
\beq
S_{ev}&=&C_{ev}T(q^+_r)^\beta(q_{vs}-q_v)^+\,,\qquad \beta \in (0,1],\label{Sev}\\
S_{cr}&=&C_{cr} q_c q_r\,\label{Scr}\qquad \\
S_{ac}&=&C_{ac} (q_c-q_{ac}^*)^+,\label{Sac}
\eeq
where $C_{ev},C_{cr},C_{ac}$ are dimensionless rate constants. Moreover $(g)^+=\max\{0,g\}$
and the constant $q_{ac}^*\geq0$ denotes the threshold for cloud water mixing ratio beyond which autoconversion of cloud water into precipitation becomes active.
The exponent $\beta$ in the evaporation term $S_{ev}$ in the literature typically appears to be chosen as $\beta \approx 0.5$, see e.g.\,\cite{KM,GS} and references therein.

As we will see below, exponents $\beta \in (0,1)$ cause difficulties in the analysis, in particular for the uniqueness of the solutions.
In a forthcoming paper we will also demonstrate how in the case $\beta= 1$, which is used in the modelling of precipitating clouds e.g.\,in \cite{DSM,HMSS,MXM}, existence and uniqueness of solutions can be proven without diffusivity imposing higher regularity on the initial data.

%We note that in the literature the source term $S_{cr}$ appears in the form $S_{cr}=C_{cr} q_c (q^+_r)^\gamma$, where either  $\gamma=1$ as above and in \cite{KM} or $\gamma$ is less but close to 1, see e.g. \cite{GS}.

We shall use the closure of the condensation term in a similar fashion to \cite{KM}
\beq\label{Scd}
S_{cd}=C_{cd}(q_v-q_{vs})q_c  +C_{cn}(q_v-q_{vs})^+\,,
\eeq
which is in the literature often defined implicitely via the equation of water vapor at saturation, see e.g.\,\cite{GS}.

\subsection{The dynamics in pressure coordinates}

Since the density of air varies strongly throughout the troposphere, the incompressibility assumption is only justified when describing shallow phenomena, and thus in general the full compressible governing equations need to be considered. However, under the assumption of hydrostatic balance
\[\frac{\pa p}{\pa z} = - g \r \,,\]
which in particular guarantees the pressure to decrease monotonically in height, the pressure can be used as the vertical coordinate.
This has the main advantage that the continuity equation takes the form of  the  incompressibility condition
\beq\label{div}
\pa_x u + \pa_y v + \pa_p \omega=0\,\qquad \textnormal{where} \quad \omega =\frac{dp}{dt}\,,
\eeq
see Lions et al.\,\cite{LTW} and  Petcu et al.\,\cite{PTZ}.

We therefore reformulate the dynamics in pressure coordinates
\[(x,y,p)\]
using hereafter the notation
\beqs
%\mathbf{v}=(u,v,\omega)\,,\qquad &\nabla = (\pa_x,\pa_y,\pa_p)\,,\qquad &\Delta =\pa_x^2+\pa_y^2+\pa_p^2\,,\\
 \mathbf{v}_h=(u,v)\,,\qquad &\nabla_h=(\pa_x,\pa_y)\,,\qquad & \Delta_h=\pa_x^2+\pa_y^2\,.
\eeqs
Due to the different units of the horizontal and the vertical derivatives in pressure coordinates, it is inappropriate to combine them into a single gradient operator. The same holds for the velocity components having different units in vertical and horizontal directions. Nevertheless the total derivative  in pressure coordinates reads
\beq\label{tot.p}
\frac{D}{Dt}&=&\pa_t + \mathbf{v}_h\cdot \nabla_h + \omega\pa_p \,.
\eeq
For the closure of the turbulent and molecular transport terms we use
\beq\label{tot.diff}
{\cal D}^* &=& \mu_*\Delta_h+ \nu_*\pa_p\left(\left(\frac{g p }{R \bar T}\right)^2\pa_p\right)
\eeq
where $\bar T=\bar T(p)$ corresponds to some background distribution being uniformly bounded from above and from below away from 0. The operator ${\cal D}^*$  thereby provides a close approximation to the full Laplacian in cartesian coordinates, see also \cite{LTW,PTZ}. This linearization around the reference profile $\bar{T}$ is also applied to the vertical transport term of $q_r$ after replacing the density using the ideal gas law, such that the moisture equations in  pressure coordinates $(x,y,p)$ with corresponding velocities $(u,v,\omega)$ become
\beq
\label{eq.qv}\frac{D q_v}{Dt}  &=& S_{ev} - S_{cd}+\cal{D}^{q_v} q_v\,,\\
\label{eq.qc}\frac{D q_c}{Dt} &=& S_{cd} - S_{ac}- S_{cr}+\cal{D}^{q_c} q_c\,,\\
\label{eq.qr}\frac{D q_r}{Dt}+V\pa_p\left(\frac{p}{\bar{T}} q_r\right) &=& S_{ac} + S_{cr}- S_{ev}+\cal{D}^{q_r} q_r\,,
\eeq
where according to \eqref{tot.p}
\beq
\frac{D q_j}{Dt}=\pa_t q_j+ u\pa_xq_j+v\pa_y q_j+ \omega\pa_pq_j\eeq
and according to \eqref{tot.diff}
\beq {\cal D}^{q_j}q_j =\mu_{q_j}(\pa_x^2 q_j+\pa_y^2q_j)+ \nu_{q_j}\pa_p\left(\left(\frac{g p }{R \bar T}\right)^2\pa_p q_j\right)\,.\eeq
The closure of the source terms thereby remains unchanged. The temperature equation in pressure coordinates reads
\beq
\label{eq.T}\frac{DT}{Dt}-\frac{R T}{c_p p} \omega=\frac{L}{c_p}(S_{cd}-S_{ev})+ {\cal D}^TT\,,
\eeq
which can again be reformulated in terms of the potential temperature as follows
\beq
\label{eq.theta}\frac{D\theta}{Dt} = \frac{\theta}{T}\left(S_T+ {\cal D}^TT\right) =
\left(\frac{p_0}{p}\right)^{\frac{R}{c_p}}\frac{L}{c_p}(S_{cd}-S_{ev})+ \widetilde{\cal D}^\theta \theta\,,
\eeq
where the diffusion in vertical direction however deviates due to the additional pressure function arising when replacing $T$ in terms of $\theta$ according to \eqref{PT}
\beq
\widetilde{\cal D}^\theta \theta=\mu_T\Delta_h\theta+
\nu_T\left(\frac{p_0}{p}\right)^{R/c_p}\pa_p\left(\left(\frac{g p}{R\bar T}\right)^{2}\pa_p\left(\left(\frac{p}{p_0}\right)^{R/c_p}\theta \right)\right)\,.
\eeq
\section{Formulation of the problem and main result}\label{sec.main}
As in Coti Zelati et al.\,\cite{CZF} we let ${\M}$ be a cylinder of the form
\beq
{\M}=\{(x,y,p): (x,y) \in {\M}', p\in (p_1,p_0)\}\,,
\eeq
where ${\M}'$ is a smooth bounded domain in $\mathbb{R}^2$ and $p_0>p_1>0$. The boundary is given by
\beq
&&\Gamma_0=\{(x,y,p)\in \overline{{\M}}: p=p_0\}\,,  \\
&&
\Gamma_1=\{(x,y,p)\in \overline{{\M}}: p=p_1\}\,,\\
&&\Gamma_{\ll} = \{(x,y,p)\in \overline{{\M}}: (x,y)\in \pa {\M}', p_0\geq p\geq p_1\}\,.
\eeq
The boundary conditions read
\beq
\label{bound.0}&&\Gamma_0:\ \pa_pT=\alpha_{0T}(T_{b 0}-T)\,,\qquad \pa_pq_{j}=\alpha_{0j} (q_{b 0 j}-q_j) \,, \quad j\in\{v,c,r\}\,,\qquad\\
&&\Gamma_1:\ \pa_pT=0\,,\qquad \pa_pq_j=0\,,\quad j\in \{v,c,r\}\,, \label{ass.qr} \\
\label{bound.ll}&&\Gamma_{\ll}:\ \pa_nT=\alpha_{{\ll}T}(T_{b \ll}-T)\,,\qquad \pa_nq_{j}=\alpha_{{\ll} j} (q_{b {\ll} j}-q_j) \,, \quad j\in\{v,c,r\}\,,
\eeq
where the multipliers $\alpha_{0T}, \alpha_{0j}, \alpha_{\ll T}, \alpha_{\ll j}$ are nonnegative quantities, which are usually constants, but for the sake of generality are allowed to vary on the corresponding boundaries and also with time.
The functions $T_{b0}, q_{b0j}, T_{b\ll}, q_{b\ll j}$ are typical temperature and moisture profiles, which are again defined on the corresponding boundaries and may vary additionally in time.
Note that the special case of $\alpha_{\ll T}=\alpha_{{\ll} v}=\alpha_{{\ll} c}=\alpha_{{\ll} r}=0$ in (\ref{bound.0})--(\ref{bound.ll}) reduces to a similar setting used for the analysis of the moisture model in \cite{CZF}. We assume all given functions $\alpha_{0 j }, \alpha_{\ll j}, \alpha_{0 T }, \alpha_{\ll T}$ as well as
$T_{b 0 }, T_{b \ll},  q_{b 0 j }, q_{b \ll j }$ to be nonnegative, sufficiently smooth and uniformly bounded, say be in $C^1$.
We assume the velocity field to be given and to satisfy
\beq\label{reg.v}
u,v,\omega\in L^\infty(0,\mathcal T; L^2(\mathcal M))\cap L^2(0,\mathcal T; H^1(\mathcal M))\cap L^r(0,\mathcal T; L^q(\mathcal M)),
\eeq
for any finite time $\mathcal T$, and for some  $r\in[2,\infty]$ and $\sigma\in[3,\infty]$ with
\beq
\frac2r+\frac3\sigma<1\,,
\eeq
satsifying the well-known Prodi-Serrin regularity condition \cite{Se}, \cite{Pr}.
Moreover we assume mass conservation, which reduces in pressure coordinates as mentioned above  to
\beq\label{inc.v}
\nabla_h\cdot \mathbf{v}_h + \pa_p\omega =0 \qquad \textnormal{in} \ \ \M\,,
\eeq
as well as the no-flux conditions on the boundary of the domain
\beq\label{noflux.v}
\mathbf{v}_h\cdot \mathbf{n}_h + \omega\, n_p=0\,\qquad \textnormal{on} \ \ \pa\M\,.
\eeq
In the following we use the abbreviation
\beq\|f\|=\|f\|_{L^2(\M)}\,,\qquad \|f\|_{L^p}=\|f\|_{L^p(\M)}\,.\eeq
According to the weight in the vertical diffusion terms, we also  introduce the weighted norms
\beq\|f\|_w=\Big\|\left(\frac{g p}{R\bar T}\right) f\Big\|\,,\qquad \|f\|^2_{H_w^1}=\|f\|^2+\|\nabla_h f\|^2+\|\pa_p f\|_w^2\,,\eeq
where we emphasize, that since the weight $\frac{gR}{p\bar T}$ is uniformly bounded from above and below by positive constants, the $H_w^1(\M)$-norm is equivalent to the $H^1(\M)$-norm.

We summarize the results obtained below in the following theorem:
\begin{thm}
\label{thm}
Let the initial data  $(T_0,q_{v0},q_{c0},$
$q_{r0}) \in (L^\infty(\M))^4\cap (H^1(\mathcal M))^4$  be nonnegative and let the
given velocity field $(\mathbf{v}_h,\omega)$ satisfy the above
assumptions \eqref{reg.v}--\eqref{noflux.v}. Then, for any (arbitrarily
large) $\T\in(0,\infty)$, there exists a unique nonnegative solution
$(T,q_{v},q_{c},q_{r})$, with initial value $(T_0,q_{v0},q_{c0},
q_{r0})$, of the boundary value problem
\eqref{eq.qv}--\eqref{eq.T} subject to
\eqref{bound.0}--\eqref{bound.ll}, on $\mathcal M\times(0,\mathcal T)$,
satisfying
\begin{eqnarray}
 (T,q_{v},q_{c},q_{r})\in L^\infty(0,\mathcal T; L^\infty(\mathcal M))\cap L^2(0,\mathcal T; H^2(\mathcal M)),\\
 (T,q_{v},q_{c},q_{r})\in C([0,\mathcal T]; H^1(\mathcal M)),\quad (\partial_tT,\partial_tq_{v},\partial_tq_{c},\partial_tq_{r})\in L^2(0,\mathcal T; L^2(\mathcal M)).
\end{eqnarray}
\end{thm}

\section{Local well-posedness and a priori estimates}\label{sec.comp}

In this section we prove the local existence and some a priori estimates
of strong solutions to system (\ref{eq.qv})--(\ref{eq.T}), or equivalently system (\ref{eq.qv})--(\ref{eq.qr}) with (\ref{eq.theta}), subject to the boundary conditions (\ref{bound.0})--(\ref{bound.ll}).

Local existence of strong solutions is stated in the following proposition.

\begin{prop}\label{prop.local}
Let the initial data $(T_0,q_{v0},q_{c0},$
$q_{r0}) \in  (H^1(\mathcal M))^4$  be nonnegative and let the
given velocity field $(\mathbf{v}_h,\omega)$ satisfy the above
assumptions \eqref{reg.v}--\eqref{noflux.v}. Then there exists a
unique local solution $(q_v, q_c, q_r, \theta)$, which depends continuously on the initial data, in some short time
interval $(0,\mathcal T_0)$, to system
\eqref{eq.qv}--\eqref{eq.qr} with (\ref{eq.theta}), subject to
\eqref{bound.0}--\eqref{bound.ll}, with initial data $(q_{v0}, q_{c0},
q_{r0}, \theta_0)$, satisfying
\begin{eqnarray}
 (q_{v},q_{c},q_{r},\theta)\in C([0,\mathcal T_0]; H^1(\mathcal M))\cap L^2(0,\mathcal T_0; H^2(\mathcal M)),\\
 (\partial_tq_{v},\partial_tq_{c},\partial_tq_{r},\partial_t\theta)\in L^2(0,\mathcal T_0; L^2(\mathcal M)).
\end{eqnarray}
\end{prop}

\begin{proof}
The well-posedness, i.e., the uniqueness and continuous dependence on initial data, will be postponed to section \ref{sec.uni}, below. Therefore, we only prove the existence part here. We are going to prove the local existence by iteration procedure, that is we construct a
sequence of vector fields $\{(q_v^n, q_c^n, q_r^n, \theta^n)\}_{n=0}^\infty$,
and show that this sequence converges to a vector field $(q_v, q_c,
q_r, \theta)$, which is the desired strong solution. Let
$\mathcal T_0\leq1$ be a
positive number to be determined later.
For simplicity of notations, we set $\mathbf{U}=(q_v, q_c, q_r, \theta)$,
and denote
$X_{\mathcal T_0}:=C([0,{\mathcal T_0}]; H^1(\mathcal M))\cap L^2(0,{\mathcal T_0}; H^2(\mathcal M))$.

Set $\mathbf{U}^0=0$, and define $\mathbf{U}^{n+1}=(q_v^{n+1}, q_c^{n+1}, q_r^{n+1}, \theta^{n+1})$, $n=0,1,\dots,$ to be the unique strong solution to the linear parabolic system
\begin{eqnarray}
  \partial_tq_v^{n+1}-\mathcal D^{q_v}q_v^{n+1}&=&S_{q_v}^n-\textbf{v}_h\cdot\nabla_h q_v^n-\omega\pa_p q_v^n,\label{Loc1}\\
  \partial_tq_c^{n+1}-\mathcal D^{q_c}q_c^{n+1}&=&S_{q_c}^n-\textbf{v}_h\cdot\nabla_h q_c^n-\omega\pa_p q_c^n,\label{Loc2}\\
  \partial_tq_r^{n+1}-\mathcal D^{q_r}q_r^{n+1}&=&S_{q_r}^n-\textbf{v}_h\cdot\nabla_h q_r^n-\omega\pa_p q_r^n-V\partial_p\left(\frac{p}{\bar T}q_r^n\right),\label{Loc3}\\
  \partial_t\theta^{n+1}-\mathcal {\tilde D}^{\theta}\theta^{n+1}&=&S_{\theta}^n-\textbf{v}_h\cdot\nabla_h \theta^n-\omega\pa_p \theta^n,\label{Loc4}
\end{eqnarray}
subject to the boundary conditions (\ref{bound.0})--(\ref{bound.ll}) and the initial condition
\beq
(q_v^{n+1}, q_c^{n+1}, q_r^{n+1}, \theta^{n+1})|_{t=0}=(q_{v0}, q_{c0}, q_{r0}, \theta_0).
\eeq
Here, $S_{q_v}^{n}, S_{q_c}^n, S_{q_r}^n$ and $S_\theta^n$ are the
source terms, respectively, expressed by $S_{ev}^n-S_{cd}^n$, $S_{cd}^n-S_{ac}^n-S_{cr}^n$, $S_{ac}^n+S_{cr}^n-S_{ev}^n$ and
$(\frac{p_0}{p})^{\frac{R}{c_p}}\frac{L}{c_p}(S_{cd}^n-S_{ev}^n)$,
with $S_{ev}^n, S_{cr}^n, S_{ac}^n$ and $S_{cd}^n$ given by (\ref{Sev})--(\ref{Scd}), by replacing $(q_v, q_c, q_r, \theta)$ with
$(q_v^n, q_c^n, q_r^n, \theta^n)$ there.

For simplicity of notations, we denote the vector field of source term $\mathbf{S}_U^n$ as
\beq
\mathbf{S}_U^n=(S_{q_v}^n, S_{q_c}^n, S_{q_r}^n, S_\theta^n).
\eeq
Recalling the expressions of $S_{ev}, S_{cr}, S_{ac}$ and $S_{cd}$
in (\ref{Sev})--(\ref{Scd}) and the regularity assumptions on
$(\textbf{v}_h, \omega)$ in (\ref{reg.v}), one can check
(see the calculations (\ref{3.5-0})--(\ref{5.30.1}), below) by
the H\"older and Sobolev embedding inequalities
that $\mathbf{S}_U^n, (\mathbf{v}_h\cdot\nabla_h \mathbf{U}^n + \omega \pa_p  \mathbf{U}^n) \in
L^2(Q_{\mathcal T_0})$, for any $\mathbf{U}^n\in X_{\mathcal T_0}$, where $Q_{\mathcal T_0}=\mathcal M\times(0,\mathcal T_0)$.
Thus, by Corollary \ref{corexistparabolic} in section \ref{secellipticparabolic}, below, $(q_v^{n+1}, q_c^{n+1}, q_r^{n+1}, \theta^{n+1})$ is well-defined and satisfies
\begin{eqnarray}
\|\mathbf{U}^{n+1}\|_{X_{\mathcal T_0}}+\|\partial_t \mathbf{U}^{n+1}\|_{L^2(Q_{\mathcal T_0})}
&\leq& C\bigg(\|\mathbf{S}_{U}^n\|_{L^2(Q_{\mathcal T_0})}+\|\mathbf{v}_h\cdot\nabla_h \mathbf{U}^n\|_{L^2(Q_{\mathcal T_0})}\\
&&+
\|\omega\pa_p \mathbf{U}^n\|_{L^2(Q_{\mathcal T_0})}+\left\|\partial_p\Big(\frac{p}{\bar T}q_r^n\Big)\right\|_{L^2(Q_{\mathcal T_0})}+1\bigg),\nonumber
\end{eqnarray}
for a positive constant $C$, which is independent of $n$ and $\mathcal T_0$, while $\mathcal T_0\in (0,1]$.

Noticing that $| \mathbf{S}_{U}^n|\leq C(| \mathbf{U}^n|+| \mathbf{U}^n|^3), $
it follows from the Young and Sobolev embedding inequalities that
\begin{eqnarray}
\|\mathbf{S}_{U}^n\|_{L^2(Q_{\mathcal T_0})}&\leq&C(\|\mathbf{U}^n\|_{L^2(Q_{\mathcal T_0})}+\| \mathbf{U}^n\|_{L^6(Q_{\mathcal T_0})}^3)\nonumber\\
&\leq& C(\| \mathbf{U}^n\|_{L^2(0,\mathcal T_0; L^2(\mathcal M))}+\| \mathbf{U}^n\|_{L^6(0,\mathcal T_0; H^1(\mathcal M))}^3)\nonumber\\
&\leq&C\sqrt{\mathcal T_0}(\| \mathbf{U}^n\|_{L^\infty(0,\mathcal T_0; L^2(\mathcal M))}+\| \mathbf{U}^n\|_{L^\infty(0,\mathcal T_0; H^1(\mathcal M))}^3)\nonumber\\
&\leq&C\sqrt{\mathcal T_0}(1+\|\mathbf{U}^n\|_{X_{\mathcal T_0}}^3),\label{3.5-0}
\end{eqnarray}
for a constant $C$, which is independent of $n$ and $\mathcal T_0$, while $\mathcal T_0\in (0,1]$. Recalling the regularity assumption in (\ref{reg.v}), it follows from the H\"older and Sobolev inequalities that
\begin{align}
\|\mathbf{v}_h\cdot\nabla_h&  \mathbf{U}^n\|_{L^2(Q_{\mathcal T_0})}
\leq
\|\mathbf{v}_h\|_{L^{\frac{2\sigma}{\sigma-3}}(0,\mathcal T_0; L^\sigma(\mathcal M))}\|\nabla_h  \mathbf{U}^n\|_{L^{\frac{2\sigma}{3}}(0,\mathcal T_0; L^{\frac{2\sigma}{\sigma-2}}(\mathcal M))}\nonumber\\
  \leq&C\mathcal T_0^{\frac{\sigma-3}{2\sigma}-\frac1r}\|\textbf{v}_h\|_{L^r(0,\mathcal T_0; L^\sigma(\mathcal M))}\|\nabla_h \mathbf{U}^n\|_{L^\infty(0,\mathcal T_0; L^2(\mathcal M))}^{1-\frac3\sigma}\|\nabla_h \mathbf{U}^n\|_{L^2(0,\mathcal T_0; L^6(\mathcal M))}^{\frac3\sigma}\nonumber\\
  \leq&C\mathcal T_0^{\frac{\sigma-3}{2\sigma}-\frac1r}\|\textbf{v}_h\|_{L^r(0,\mathcal T_0; L^\sigma(\mathcal M))}\|\mathbf{U}^n\|_{L^\infty(0,\mathcal T_0; H^1(\mathcal M))}^{1-\frac3\sigma}\|\mathbf{U}^n\|_{L^2(0,\mathcal T_0; H^2(\mathcal M))}^{\frac3\sigma}\nonumber\\
  \leq& C\mathcal T_0^{\frac{\sigma-3}{2\sigma}-\frac1r}\|\textbf{v}_h\|_{L^r(0,\mathcal T; L^\sigma(\mathcal M))}\|\mathbf{U}^n\|_{X_{\mathcal T_0}}\leq C\mathcal T_0^{\frac{\sigma-3}{2\sigma}-\frac1r}\|\mathbf{U}^n\|_{X_{\mathcal T_0}},\label{5.30.1}
\end{align}
for a positive constant $C$, which is independent of $n$ and $\mathcal T_0$, while $\mathcal T_0\in (0,1]$; with an analogous estimate that can be established for $\|\omega\pa_p  \mathbf{U}^n\|_{L^2(Q_{\mathcal T_0})}$. Moreover,
\beq
\left\|\partial_p\left(\frac{p}{\bar T}q_r^n\right)\right\|_{L^2(Q_{\mathcal T_0})}\leq C\|q_r^n\|_{L^2(0,\mathcal T_0; H^1(\mathcal M))}\leq C\sqrt\mathcal T_0 \|q_r^n\|_{X_{\mathcal T_0}},
\eeq
for a positive constant $C$, which is independent of $n$ and $\mathcal T_0$, while $\mathcal T_0\in (0,1]$.

Therefore, we have
\begin{eqnarray}
\|\mathbf{U}^{n+1}\|_{X_{\mathcal T_0}}+\|\partial_t\mathbf{U}^{n+1}\|_{L^2(Q_{\mathcal T_0})}&\leq&
 C[\mathcal T_0^\delta(\|\mathbf{U}^n\|_{X_{\mathcal T_0}}+\|\mathbf{U}^n\|_{X_{\mathcal T_0}}^3)+1]\nonumber\\
&\leq& C_1(\mathcal T_0^\delta\|\mathbf{U}^n\|_{X_{\mathcal T_0}}^3+1),\label{Loc5}
\end{eqnarray}
for a positive constant $C_1$, which is independent of $n$ and $\mathcal T_0$, while $\mathcal T_0\in (0,1]$; and $\delta=\min\{\frac12,\frac{\sigma-3}{2\sigma}-\frac1r\}=\frac{\sigma-3}{2\sigma}-\frac1r>0$.
Set $M=2C_1$ and $\mathcal T_0=\min\left\{1,(2C_1)^{-\frac3\delta}\right\}$.
Then, thanks to (\ref{Loc5}), and recalling that $\mathbf{U}^0=0$, one can easily show by induction that
\beq
\|\mathbf{U}^{n+1}\|_{X_{\mathcal T_0}}+\|\partial_t\mathbf{U}^{n+1}\|_{L^2(Q_{\mathcal T_0})}\leq M,\quad n=0,1,2,\dots.\label{3.14li}
\eeq

 Next, we show that  $\{U^n\}_{n=1}^\infty$ is a Cauchy sequence in the space $C([0,\mathcal T_0^*]; L^2(\mathcal M))$, for a positive number $\mathcal T_0^*\in(0,\mathcal T_0)$. As we have pointed out in the Introduction, if the exponent  in the evaporation term $S_{ev}$, $\beta \in (0,1)$, we need to introduce the following  new unknowns
$$
Q^n=q_v^n+q_r^n,\quad H^n=T^n-\frac{L}{c_p}(q_c^n+q_r^n),\quad n=0,1,2,\cdots.
$$
To show that $\{U^n\}_{n=1}^\infty$ is a Cauchy sequence in the space $L^\infty(0,\mathcal T_0^*; L^2(\mathcal M))$, instead of carrying out the $L^2$-estimate for the difference $U^{n+1}-U^n$, we will perform
the corresponding estimate for $(Q^{n+1}-Q^n,q_c^{n+1}-q_c^n,q_r^{n+1}-q_r^n,H^{n+1}-H^n)$.

To simplify the  notation, we set
\begin{eqnarray*}
\alpha_n(t)&=&1+\|q_c^{n-1}\|_{L^\infty}^2+\|q_c^{n}\|_{L^\infty}^2+\| q_r^{n-1}\|_{L^\infty}^2+\|q_r^n\|_{L^\infty}^2+\|q_v^n\|_{L^\infty}^2 +\|\theta^n\|_{L^\infty}^2,\\
  \phi_n(t)&=&\|Q^n-Q^{n-1}\|^2+\|q_{r}^n-q_{r}^{n-1}\|^2+\|q_{c}^n-q_{c}^{n-1}\|^2+ \|H^n-H^{n-1}\|^2,
\end{eqnarray*}
for $n=1,2,\cdots.$ Then, thanks to (\ref{3.14li}), and using the Gagliardo-Nirenberg inequality, $\|f\|_{L^\infty(\mathbb R^3)}\leq C\|f\|_{L^2(\mathbb R^3)}^{\frac14} \|f\|_{H^2(\mathbb R^3)}^{\frac34}$, and the H\"older inequality, one deduces
\begin{eqnarray}
  \int_0^{\mathcal T_0^*}\alpha_n(t)dt &\leq& \mathcal T_0^*+
  C\sum_{f\in\{q_c^{n-1}, q_c^{n}, q_r^{n-1}, q_r^n, q_v^n,\theta^n\}}\int_0^{\mathcal T_0^*}\|f\|_{L^2}^{\frac12} \|f\|_{H^2}^{\frac32}dt\nonumber\\
  &\leq&\mathcal T_0^* +C(\mathcal T_0^*)^{\frac14}M^{\frac12} \sum_{f\in\{q_c^{n-1}, q_c^{n}, q_r^{n-1}, q_r^n, q_v^n,\theta^n\}}\left(\int_0^{\mathcal T_0^*} \|f\|_{H^2}^2dt\right)^{\frac34}\nonumber\\
  &\leq&\mathcal T_0^* +CM^2(\mathcal T_0^*)^{\frac14}. \label{3.18li}
\end{eqnarray}

We now perform
the $L^2$-estimate for $(Q^{n+1}-Q^n,q_c^{n+1}-q_c^n,q_r^{n+1}-q_r^n,H^{n+1}-H^n)$. Since the proof
is similar to that of showing the uniqueness in Proposition \ref{prop.uni},
below, we only sketch the proof here.
First, for the estimate of $Q^{n+1}-Q^n$, following the same arguments for
deriving (\ref{4.13-li}), one ends up with an energy inequality, which is the
same as (\ref{4.13-li}), but replacing $Q_1-Q_2$ and $q_{r1}-q_{r2}$ by $Q^{n+1}-Q^n$ and $q_r^{n+1}-q_r^n$, respectively, on the left and the first
line on the right of (\ref{4.13-li}), and replacing $q_{r1}, q_{r2}, Q_1, Q_2, q_{c1}, q_{c2}, q_{v1}, q_{v2}, T_1, T_2$ by $q_r^n, q_r^{n-1}, Q^n, Q^{n-1}$, $q_c^n$, $q_c^{n-1}$, $q_v^n, q_v^{n-1}, T^n, T^{n-1}$, respectively, in the other terms on the right of (\ref{4.13-li}). Then, by using (\ref{06202}), below, as well as the Young inequality, we have
\begin{eqnarray}
  &&\frac12\frac{d}{dt}\|Q^{n+1}-Q^n\|^2+\frac{\mu_{q_c}}{4}\|\nabla_h(Q^{n+1} -Q^n) \|^2+\frac{\nu_{q_v}}{4}\|\partial_p(Q^{n+1}-Q^n)\|_w^2\nonumber\\
  &\leq&C_Q(\mu_{q_r}\|q_r^{n+1}-q_r^n\|^2+\nu_{q_r} \|\partial_p(q_r^{n+1}-q_r^n)\|_w^2)+C[\phi_{n+1}(t)+\alpha_n(t)\phi_n(t)]. \label{9261LI}
\end{eqnarray}
Next, for the estimate of $q_r^{n+1}-q_r^n$, similar as before, following the same arguments for deriving (\ref{4.17-li}), one obtains a similar energy inequality
as (\ref{4.17-li}), from which, by using the monotonicity in the evaporation term $[(q_r^n)^\beta-(q_r^{n-1})^\beta](q_r^n-q_r^{n-1})\geq0$, as well as the Young inequality, one has the following inequality
\begin{align}
\frac12\frac{d}{dt}\|q_r^{n+1}-q_r^n\|^2&+\mu_{q_r}\|\nabla_h(q_r^{n+1}-q_r^n) \|^2+\frac{\mu_{q_r}}{2}\|\partial_p(q_r^{n+1}-q_r^n)\|_w^2\nonumber\\
  \leq&C[\phi_{n+1}(t)+\alpha_n(t)\phi_n(t)]. \label{9262}
\end{align}
And finally, similar arguments as for (\ref{est.qc.uni}) and (\ref{est.H.uni}) yield the estimates for $q_c^{n+1}-q_c^n$ and $H^{n+1}-H^n$ as follows
\begin{eqnarray}
  \frac12\frac{d}{dt}\|q_c^{n+1}-q_c^n\|^2+\mu_{q_c}\|\nabla_h(q_c^{n+1}-q_c^n)\|^2 +\nu_{q_c}\|\partial_p(q_c^{n+1}-q_c^n\|_w^2\leq C\alpha_n(t)\phi_n(t), \label{9263LI}
\end{eqnarray}
and
\begin{eqnarray}
  &&\frac12\frac{d}{dt}\|H^{n+1}-H^n\|^2+\frac{\mu_T}{2}\|\nabla_h(H^{n+1}-H^n)\|^2 +\frac{\nu_T}{2}\|\partial_p(H^{n+1}-H^n)\|_w^2\nonumber\\
  &\leq&C_H\Big(\mu_{q_c}\|\nabla_h(q_c^{n+1}-q_c^n)\|^2+\nu_{q_c}\|\partial_p (q_c^{n+1}-q_c^n)\|_w^2+\mu_{q_r}\|\nabla_h(q_r^{n+1}-q_r^n)\|^2\nonumber\\
  &&+\nu_{q_r}\| \partial_p(q_r^{n+1}-q_r^n)\|_w^2\Big)+C[\phi_{n+1}(t)+\alpha_n(t)\phi_n(t)]. \label{9264LI}
\end{eqnarray}

Set
$$
J_n(t) =\frac{1}{2C_Q} \|Q^n-Q^{n-1}\|^2+\|q_{r}^n-q_{r}^{n-1}\|^2+\|q_{c}^n-q_{c}^{n-1}\|^2+ \frac{1}{2C_H} \|H^n-H^{n-1}\|^2.
$$
Noticing that $\phi_n(t)\leq CJ_n(t)$, it follows from (\ref{9261LI})--(\ref{9264LI}) that
$$
\frac{d}{dt}J_{n+1}(t)\leq C[J_{n+1}(t)+\alpha_n(t)J_n(t)].
$$
Since $J_{n+1}(0)=0$ it follows from the above, by virtue  the Gronwall inequality, that
$$
J_{n+1}(t)\leq e^{Ct}\int_0^t\alpha_n(s)J_n(s)ds, \quad \hbox{for all} \quad t\in [0,\mathcal T_0^*] \subset [0,\mathcal T_0].
$$
Thanks to the above, recalling (\ref{3.18li}), and choosing $\mathcal T_0^*$ small enough, we have
$$
\sup_{0\leq t\leq\mathcal T_0^*}J_{n+1}(t)\leq e^{C\mathcal T_0^*}\left(\int_0^{\mathcal T_0^*}\alpha_n(t)dt\right)\sup_{0\leq t\leq\mathcal T_0^*}J_n(t) \leq\frac12\sup_{0\leq t\leq\mathcal T_0^*}J_n(t),
$$
and thus
$$
\sup_{0\leq t \leq\mathcal T_0^*}J_{n+1}(t)\leq \frac{C}{2^n}.
$$

Thanks to the above estimate, it is clear that $\{(Q^n,q_c^n,q_r^n,H^n)\}_{n=1}^\infty$ is a Cauchy sequence in $L^\infty(0,\mathcal T_0^*; L^2(\mathcal M))$, and consequently $\{U^n\}_{n=1}^\infty$ is a Cauchy sequence in the same space. By virtue of
this fact, and recalling the a priori estimate (\ref{3.14li}), by the Aubin-Lions lemma, there is a vector field $\mathbf{U}\in X_{\mathcal T_0^*}$, with $\partial_t\mathbf{U}\in L^2(Q_{\mathcal T_0^*})$, such that
\begin{eqnarray}
  \mathbf{U}^n&\rightharpoonup& \mathbf{U},\qquad \, \mbox{in }L^2(0,\mathcal T_0^*; H^2(\mathcal M)), \\
  \partial_t\mathbf{U}^n&\rightharpoonup&\partial_t\mathbf{U},\quad\mbox{ in }L^2(Q_{\mathcal T_0^*}), \\
  \mathbf{U}^n&\rightarrow &\mathbf{U},\qquad\mbox{ in }L^2(0,\mathcal T_0^*; H^1(\mathcal M))\cap C([0,\mathcal T_0^*]; L^2(\mathcal M)).
\end{eqnarray}
We point out that the above convergence holds for the whole sequence $\{U^n\}_{n=1}^\infty$, rather than only for a subsequence.

Using the above convergences, one can take the limit, as $n\rightarrow\infty$, in (\ref{Loc1})--(\ref{Loc4}) to show that $\mathbf{U}=(q_v, q_c, q_r, \theta)$ is a strong solution of the boundary value problem \eqref{eq.qv}--\eqref{eq.qr} with (\ref{eq.theta}), subject to \eqref{bound.0}--\eqref{bound.ll}, with initial data $(q_{v0}, q_{c0}, q_{r0}, \theta_0)$.
\end{proof}

In the following proposition we derive nonnegativity and uniform
boundedness for the moisture quantities and the temperature. Here the
sequence of the derivation of the individual bounds needs to be in the
right order to close the estimates consecutively.

\begin{prop}\label{prop.comp}
Let $\mathcal T\in(0,\infty)$ and $(T,q_v,q_c,q_r)$ be a solution to  \eqref{eq.qv}--\eqref{eq.T}  in $\mathcal M\times(0,\mathcal T)$ subject to the boundary conditions (\ref{bound.0})--(\ref{bound.ll}) with non-negative intial data $(T_0,q_{v 0},q_{c 0},q_{r 0})\in (L^\infty(\M))^4\cap (H^1(\mathcal M))^4$, satisfying the regularities stated in Proposition~\ref{prop.local} by replacing $\mathcal T_0$ with $\mathcal T$. Then for every $t\in [0,\mathcal T ]$ the solution $(T,q_{v},q_{c},q_{r})(t)$ satisfies
\beq
0\leq q_{v}\leq  q_v^*\,,\quad 0\leq q_{c}\leq q_c^*\,,\quad 0\leq q_{r}\leq  q_r^*\,,\quad 0\leq T\leq T^*\,,
\eeq
where
\beq\label{qv.star}
q_v^*=\max\big\{\|q_{v 0}\|_{L^\infty{(\M)}},\|q_{b0v}\|_{L^\infty((0,\T)\times\M')},\|q_{b\ll v}\|_{L^\infty((0,\T)\times\Gamma_\ll)}, q_{vs}^*\big\}
\eeq
with $q_{vs}^*$ being the constant in (\ref{1.3-1}), and $q_c^*,q_r^*,T^*$ are constants depending on the following quantities:
\beq
&&q_c^*\, =C_{q_c}\big(\T, \|q_{c 0}\|_{L^\infty{(\M)}},\|q_{b0c}\|_{L^\infty((0,T)\times\M')},\|q_{b\ll c}\|_{L^\infty((0,T)\times\Gamma_\ll)},q_v^*,q_{vs}^*\big)\,,\\
&&q_r^*\, =C_{q_r}\big(\T, \|q_{r 0}\|_{L^\infty{(\M)}},\|q_{b0r}\|_{L^\infty((0,T)\times\M')},\|q_{b\ll r}\|_{L^\infty((0,T)\times\Gamma_\ll)},q_c^*\big)\,,\\
&&T^*=C_{T}\big(\T,\|T_0\|_{L^\infty{(\M)}},\|T_{b0}\|_{L^\infty((0,T)\times\M')}, \|T_{b\ll}\|_{L^\infty((0,T)\times\Gamma_\ll)},q_v^*,q_c^*,q_{vs}^*\big)\,.
\eeq
\end{prop}
\begin{proof}
(i)\underline{\emph{Nonnegativity of $q_v, q_c, q_r$ and $\theta$.}} For deriving this first part of the comparison principles we
employ the Stampacchia method and therefore test the equations of the
mixing ratios with their negative parts, where in the following we use
\beq
f=f^+-f^-\,\eeq
for splitting a function $f$ into its
positive and negative parts, with $f^+=\max\{f,0\}$ and $f^-=\max\{-f,0\}$.

For the aim of the later uses, we first carry out some calculations on
the integrals over the domain $\mathcal M$ of the products of the
diffusion and convection terms with $q_j^-$, $j\in\{v, c, r\}$. Integration
by parts and using the boundary conditions (\ref{bound.0})--(\ref{bound.ll}),
one deduces
\begin{eqnarray}
\int_\M q_j^-{\cal D}^{q_j}q_jd\M&=&\int_\M\left[\mu_{q_j}\Delta_hq_j+\nu_{q_j}\partial_p
\left(\left(\frac{gp}{R\bar T}\right)^2\partial_pq_j\right)\right]q_j^-d\M\nonumber\\
&=&\mu_{q_j}\int_{\Gamma_\ll}(\partial_nq_j)q_j^-d\Gamma_\ll+\nu_{q_j}\left.\int_{\M'}
\left(\frac{gp}{R\bar T}\right)^2(\partial_pq_j)q_j^-d\M'\right|_{p_1}^{p_0}\nonumber\\
&&-\int_\M\left[\mu_{q_j}\nabla_hq_j\cdot\nabla_hq_j^-+\nu_{q_j}\left(\frac{gp}{R\bar T}\right)^2\pa_pq_j\partial_pq_j^-\right]d\M\nonumber\\
&=&\mu_{q_j}\|\nabla_hq_j^-\|^2+\nu_{q_j}\|\partial_pq_j^-\|_w^2+\mu_{q_j} \int_{\Gamma_\ll}\alpha_{\ll j}(q_{\ll j}-q_j)q_j^-d\Gamma_\ll\nonumber\\
&&+\nu_{q_j}\int_{\M'}\left(\frac{gp_0}{R\bar T}\right)^2\alpha_{b0j}(q_{b0j}-q_j)q_j^-d\M'.
\end{eqnarray}
Since the functions $q_{\ll j}$ and $q_{0j}$ are both nonnegative and
$q_j q_j^- =- ( q_j^-)^2$ a.e., the last two boundary integrals
are nonnegative, and we obtain
\beq\label{est.diff.q}
  \int_\M q_j^-{\cal D}^{q_j}q_jd\M
\geq\mu_{q_j}\|\nabla_hq_j^-\|^2+\nu_{q_j}\|\partial_pq_j^-\|_w^2.
\eeq
The integral containing the advection term vanishes due to \eqref{inc.v} and \eqref{noflux.v}, since
\beq
 &&\int_\M (\mathbf{v}_h\cdot\nabla_h q_j + \omega \pa_p q_j) q_j^-d\M= -\frac12\int_\M (\mathbf{v}_h\cdot\nabla_h + \omega \pa_p) (q_j^-)^2 d\M\nonumber\\
&=&\label{est.trans.q} -\frac12\int_{ \pa\M}(\mathbf{v}_h\cdot\mathbf{n}_h + \omega n_p)(q_j^-)^2d(\pa\M)+\frac12\int_{ \M}(q_j^-)^2(\nabla_h\cdot \mathbf{v}_h + \pa_p\omega )d\M =0\,.\qquad
\eeq

We now proceed with the derivation of the nonnegativity of $q_v, q_c, q_r$ and $T$, where we start with the cloud water mixing ratio
$q_c$. Multiplying equation (\ref{eq.qc}) by
$-q_c^-$, recalling \eqref{est.diff.q}--\eqref{est.trans.q},
noticing that $q_c^-(q_c-q_{ac}^*)^+=0$ since $q_{ac}^*\geq0$ and applying the
Sobolev embedding
inequality, one deduces
\begin{eqnarray}
&&\frac12\frac{d}{dt}\int_\M (q_c^-)^2d\M \leq - \int_\M q_c^-(S_{cd}-S_{ac}-S_{cr}) d\M\nonumber \\
&=& -\int_\M q_c^-\big(C_{cd}(q_v-q_{vs})q_c +C_{cn}(q_v-q_{vs})^+- C_{ac}(q_{c}-q_{ac}^*)^+-C_{cr} q_{c}q_r\big)d\M\nonumber\\
&\leq&\int_{\mathcal M}(C_{cr}q_cq_r-C_{cd}(q_v-q_{vs})q_c)q_c^-d\mathcal M=\int_{\mathcal M}(C_{cd}(q_v-q_{vs})-C_{cr}q_r)(q_c^-)^2d\mathcal M\nonumber\\
&\leq& C(1+\|(q_v,q_r)\|_{L^\infty(\mathcal M)})\|q_c^-\|^2
\leq C(1+\|(q_v,q_r)\|_{H^2(\mathcal M)})\|q_c^-\|^2, \label{est.qc.0}
\end{eqnarray}
from which due to the Gronwall inequality we have
\beq
\|q_c^-\|^2(t)\leq e^{C\int_0^t(1+\|(q_v,q_r)\|_{H^2(\mathcal M)})dt}
\|q_{c0}^-\|^2=0\,,
\eeq
implying $q_c^-\equiv0$, and thus the nonnegativity of $q_c$.

We next prove the nonnegativity of $q_v$. Multiplying equation
\eqref{eq.qv} by $-q_v^-$, integrating the resultant over $\mathcal M$,
using \eqref{est.diff.q}--\eqref{est.trans.q} and noticing that
$(q_v-q_{vs})^+q_v^-=0$ due to $q_{vs}\geq0$, we get
\beq
&&\frac12\frac{d}{dt}\int_\M (q_v^-)^2d\M
\leq - \int_\M q_v^-(S_{ev}-S_{cd}) d\M \nonumber\\
&=& \int_\M\big(C_{cd}(q_v-q_{vs})q_c +C_{cn}(q_v-q_{vs})^+  - C_{ev}T(q_r^+)^\beta(q_{vs}-q_{v})^+\big)q_v^-d\M\nonumber\\
&= &\int_\M \big(C_{cd}(q_v-q_{vs})q_c- C_{ev}T(q_r^+)^\beta(q_{vs}-q_{v})^+\big)q_v^- d\M\nonumber\\
&\leq&\int_\M \big(C_{cd}(q_v-q_{vs})q_c+ C_{ev}T^-(q_r^+)^\beta(q_{vs}-q_{v})^+\big)q_v^- d\M.
\eeq
Recalling $q_{vs}\geq0,$ $q_c\geq0$ and $q_{vs}(p,T)=0$
for $T\leq0$ from (\ref{nonneg.qvs}), one can deduce
$q_{vs}q_cq_v^-\geq0$ and
$T^-(q_r^+)^\beta(q_{vs}-q_v)^+q_v^-=T^-(q_r^+)^\beta(-q_v)^+q_v^-
=T^-(q_r^+)^\beta(q_v^-)^2$. Using moreover the Sobolev and Young inequalities, one obtains
\begin{eqnarray}
  \frac12\frac{d}{dt}\|q_v^-\|^2&\leq&\int_{\mathcal M}\big(-C_{cd}q_c+C_{ev}T^-(q_r^+)^\beta\big)(q_v^-)^2d\mathcal M\nonumber\\
  &\leq&C(\|q_c\|_{L^\infty(\mathcal M)}+\|T\|_{L^\infty(\mathcal M)} \|q_r\|_{L^\infty(\mathcal M)}^\beta)\|q_v^-\|^2\nonumber\\
  \label{est.qv.min}&\leq&C(1+\|(q_c,q_r,\theta)\|_{H^2(\mathcal M)}^2)\|q_v^-\|^2\,.
\end{eqnarray}
In \eqref{est.qv.min} we made use of the fact that $\beta\in(0,1]$. Applying the Gronwall inequality to the resulting estimate, one obtains
\begin{eqnarray}
  \|q_v^-\|^2(t)\leq e^{C\int_0^t(1+\|(q_c,q_r,\theta)\|_{H^2(\mathcal M)}^2)dt}\|q_{v0}^-\|^2=0.
\end{eqnarray}
Therefore $q_v^-\equiv0$, implying again $q_v\geq0$.

We now turn to  the rain water mixing ratio $q_r$. Before proceeding in proving the nonnegativity, we show how to deal with the
integral involving the terminal velocity  by applying the Young inequality as follows
\beq
V\int_{\M}q_r^-\pa_p\left(\frac{p}{\bar{T}}q_r\right)d\M &=&
-V\int_{\M}\Big[(q_r^-)^2\pa_p\left(\frac{p}{\bar{T}}\right) +\left(\frac{p}{\bar{T}}\right)q_r^-\pa_pq_r^- \Big]d\M \nonumber\\
&\leq&C\|q_r^-\|^2+\frac{\nu_{q_r}}{2}\|\pa_pq_r^-\|^2_w.
\label{est.trans.V}
\eeq
Testing now \eqref{eq.qr} with $q_r^-$ and employing \eqref{est.diff.q}--\eqref{est.trans.V}, we obtain
\beq
&&\frac12\frac{d}{dt}\int_\M (q_r^-)^2d\M \leq C\|q_r^-\|^2- \int_\M q_r^-(S_{ac}+S_{cr}-S_{ev}) d\M \nonumber\\
&=&C\|q_r^-\|^2-\int_\M q_r^-\left(C_{ac}(q_c-q_{ac}^*)^+ +C_{cr} q_{c}q_r-C_{ev} T(q_r^+)^\beta(q_{vs}-q_v)^+\right)d\M\quad \nonumber\\
&\leq&C\|q_r^-\|^2+C_{cr}\int_{\mathcal M}q_c(q_r^-)^2d\mathcal M
\leq C(1+\|q_c\|_{L^\infty(\mathcal M)})\|q_r^-\|^2,
\eeq
where we have used $q_r^-(q_r^+)^\beta=0$. Applying the Gronwall inequality to the above inequality and using the Sobolev embedding theorem, we deduce
\begin{eqnarray}
  \|q_r^-\|^2(t)&\leq& e^{C\int_0^t(1+\|q_c\|_{L^\infty(\mathcal M)})dt} \|q_{r0}^-\|^2\nonumber\\
  &\leq&e^{C\int_0^t(1+\|q_c\|_{H^2(\mathcal M)})dt} \|q_{r0}^-\|^2=0.
\end{eqnarray}
Thus $q_r^-\equiv0$ and we obtain the desired nonnegativity $q_r\geq0$.

Finally, we prove the nonnegativity of $\theta$ and therefore  first deal with the
integrals involving the diffusion terms. Integration by parts
yields
\begin{align}
\int_\M\theta^-\,&\widetilde{\cal D}^\theta \theta d\M=
\int_\M\left[\mu_T\Delta_h\theta+\nu_T\left(\frac{p_0}{p}\right)^{\frac{R}{c_p}} \partial_p\left(\left(\frac{gp}{R\bar T}\right)^2
\partial_p\left(\left(\frac{p}{p_0}\right)^{\frac{R}{c_p}}\theta\right)
\right)\right]\theta^-d\M\nonumber\\
=&-\int_\M\left[\mu_T\nabla_h\theta\cdot\nabla_h\theta^-+\nu_T\left(\frac{gp} {R\bar T}\right)^2\partial_p\left(\left(\frac{p_0}{p}\right)^{\frac{R}{c_p}}
\theta^- \right)\partial_p\left(\left(\frac{p}{p_0}\right)^{\frac{R}{c_p}}\theta \right)\right] d\M\nonumber\\
&+\mu_T\int_{\Gamma_\ll}\theta^-\partial_n\theta d\Gamma_\ll+
\nu_T\left.\int_{\M'} \left(\frac{gp}{R\bar T}\right)^2\partial_p\left(\left(\frac{p}{p_0}\right)^{\frac{R}{c_p}}\theta\right)
\left(\frac{p_0}{p}\right)^{\frac{R}{c_p}}\theta^-d\M'\right|_{p_1}^{p_0}.
\end{align}
We recall that
$\theta=T(\frac{p_0}{p})^{R/c_p}$, as well as $T_{b\ll}\geq 0$ and $T_{b0}\geq 0$. By the boundary conditions
(\ref{bound.0})--(\ref{bound.ll})  we then have
\beq
&&\textnormal{on} \ \Gamma_\ll: \  \ \theta^-\pa_n\theta = \left(\frac{p_0}{p}\right)^{\frac{2R}{c_p}}T^-  \partial_nT=\left(\frac{p_0}{p}\right)^{\frac{2 R}{c_p}}
T^-\alpha_{\ll T}(T_{b\ll}-T)\geq 0\,,\\
&&\textnormal{on} \ \Gamma_0: \  \ \theta^-\partial_p\left(\left(\frac{p}{p_0}\right)^{R/c_p}\theta\right) =\theta^-\partial_pT=T^- \left(\frac{p_0}{p}\right)^{\frac{R}{c_p}}\alpha_{0T}(T_{b0}-T) \geq 0\,,\quad \\
&&\textnormal{on} \ \Gamma_1: \  \ \theta^-\partial_p\left(\left(\frac{p}{p_0}\right)^{R/c_p}\theta\right) =\theta^-\partial_pT=0.
\eeq
Straightforward computation of the integral containing the $p-$derivatives gives
\beq
&&\nu_T\int_\M\left(\frac{gp} {R\bar T}\right)^2\partial_p\left(\left(\frac{p_0}{p}\right)^{R/c_p} \theta^- \right)\partial_p\left(\left(\frac{p}{p_0}\right)^{R/c_p}\theta \right) d\M\nonumber\\
&& \qquad  = \nu_T \int_\M \left(- \left(\frac{gp} {R\bar T}\right)^2 (\pa_p \theta^-)^2 +\left(\frac{g}{c_p\bar T}\right)^2 (\theta^-)^2 \right)d\M\,.
\eeq
Thus, we have
\beq
\int_\M\theta^-\,\widetilde{\cal D}^\theta \theta d\M\geq \mu_T\left\|\nabla_h\theta^-\right\|^2+\nu_T\left\| \partial_p
\theta^-\right\|_w^2-\nu_T\int_{\mathcal M} \left(\frac{g}{c_p\bar T}\right)^2 (\theta^-)^2d\M\,.
\eeq
By the aid of the above, multiplying equation (\ref{eq.theta}) by $\theta^-$,
it follows from integration by parts that
\beq
\frac12\frac{d}{dt}\|\theta^-\|^2
&\leq&  - \frac{L}{c_p}\int_\M \left(\frac{p_0}{p}\right)^{\frac{R}{c_p}}\theta^-\big(C_{cd}(q_v-q_{vs})q_c+C_{cn}(q_v-q_{vs})^+q_c \big)d\M\nonumber\\
&&+ \frac{L}{c_p}\int_\M\theta^-\left(\frac{p_0}{p}\right)^{\frac{R}{c_p}} C_{ev}T(q_r^+)^\beta(q_{vs}-q_v)^+d\M\nonumber\\
&&+\nu_T\int_{\mathcal M} \left(\frac{g}{c_p\bar T}\right)^2 (\theta^-)^2d\M\leq C\|\theta^-\|^2,
\eeq
where we used the assumption (\ref{nonneg.qvs}) that $q_{vs}=0$ for $T\leq 0$ (or $\theta\leq 0$ respectively), and the nonnegativity of $q_v, q_c$ and $q_r$.
Since $\theta_0^-\equiv 0$, the Gronwall inequality implies again $\theta^-\equiv 0$ for all $t>0$, and thus $\theta\geq0$.

(ii) \underline{Boundedness of $q_v$.}
We will test the equation  \eqref{eq.qv} with $(q_v-q_v^*)^+$. For the diffusion operator we thereby proceed similar to above:
\beq
&&\int_\M (q_v-q_v^*)^+ \, {\cal D}^{q_v}q_v \,  d\M\nonumber\\
&=&\int_\M\left[\mu_{q_v}\Delta_hq_v+\nu_{q_v}\partial_p
\left(\left(\frac{gp}{R\bar T}\right)^2\partial_pq_v\right)\right](q_v-q_v^*)^+d\M\nonumber\\
&=&\mu_{q_v}\int_{\Gamma_\ll}(q_v-q_v^*)^+\partial_nq_v
d\Gamma_\ll+\nu_{q_v}\left.\int_{\M'}\left(\frac{gp}{R\bar T}\right)^2(q_v-q_v^*)^+\partial_pq_vdxdy\right|_{p_1}^{p_0}\nonumber\\
&&-\int_\M\left[\mu_{q_v}\nabla_hq_v\cdot\nabla_h (q_v-q_v^*)^+ +\nu_{q_v}\left(\frac{gp}{R\bar T}\right)^2\pa_pq_v\partial_p(q_v-q_v^*)^+\right]d\M\nonumber\\
&=&-\mu_{q_v}\|\nabla_h(q_v-q_v^*)^+\|^2-\nu_{q_v}\|\partial_p(q_v-q_v^*)^+\|_w^2\nonumber\\
&&+\mu_{q_v} \int_{\Gamma_\ll}\alpha_{\ll v}(q_{b\ll v}-q_v)(q_v-q_v^*)^+d\Gamma_\ll\nonumber\\
&&+\nu_{q_v}\int_{\M'}\left(\frac{gp_0}{R\bar T}\right)^2\alpha_{b0v}(q_{b0v}-q_v)(q_v-q_v^*)^+d\M'\bigg|_{p=p_0}.
\eeq
By the definition of $q_v^*$ we have
\beq\label{bound.est.qv}
(q_{b\ll v}-q_v)(q_v-q_v^*)^+\leq 0\quad\mbox{on }\Gamma_\ll\,,\qquad (q_{b0v}-q_v)(q_v-q_v^*)^+\leq 0\quad\mbox{on }\Gamma_0,
\eeq
implying
\beq\label{est.diff.qv}
  \int_\M(q_v-q_v^*)^+{\cal D}^{q_v}q_vd\M
\leq -\mu_{q_v}\|\nabla_h(q_v-q_v^*)^+\|^2- \nu_{q_v}\|\partial_p(q_v-q_v^*)^+\|_w^2,
\eeq
and leading further to the inequality
\beq
&&\frac12\frac{d}{dt}\int_\M ((q_v-q_v^*)^+)^2d\M \leq  \int_\M (q_v-q_v^*)^+(S_{ev}-S_{cd}) d\M \\
&= &\int_\M (q_v-q_v^*)^+\left(C_{ev}T(q_r^+)^\beta(q_{vs}-q_{v})^+-C_{cd} (q_v-q_{vs})q_c -C_{cn}(q_v-q_{vs})^+\right) d\M. \nonumber
\eeq
Thanks to the above inequality,  the definition of $q_v^*$ in \eqref{qv.star} (implying  in particular $q_v^*\geq q_{vs}$) and
the nonnegativity of $T$ and $q_c$, one has
\beq
\frac12\frac{d}{dt}\int_\M ((q_v-q_v^*)^+)^2d\M\leq0.
\eeq
Therefore, $\|(q_v-q_v^*)^+\|^2(t)\leq \|(q_{v0}-q_v^*)^+\|^2=0$, such that $q_v\leq q_v^*$.

(iii)\underline{\emph{Boundedness of $q_c,q_r$.}}
We start with the derivation of the uniform boundedness of $q_j$, with $j\in\{c,r\}$. For any $m\geq2$, we denote the cutoff function $q_{j,k_j}$ as
\beq
q_{j,k_j}=(q_j-k_j)^+ \eeq
with
\beq
k_j=\sup_{t\in[0,\T]}(\|q_{b\ll j}\|_{L^\infty(\Gamma_\ll)}+\|q_{b0j}\|_{L^\infty(\pa \M')}+\|q_{j0}\|_{L^\infty(\mathcal M)}),
\eeq
for $j\in\{c,r\}$. We will use the method of testing equations (\ref{eq.qc})--(\ref{eq.qr}) with $mq_{j,k_j}^{m-1}$, which
is typically employed for
equations with nonlinear diffusion, see e.g.\,\cite{K}. Integrating by
parts, the transport operator vanishes as before, and for the diffusion terms
we have
\beq
 \int_\M q_{j,k_j}^{m-1}{\cal D}^{q_j} q_j \, d\M &=& \mu_{q_j}\int_{\Gamma_\ll}  q_{j,k_j}^{m-1}\pa_n q_j d\Gamma_\ll + \nu_{q_j}\int_{\M'}  \left(\frac{gp}{R\bar T}\right)^2q_{j,k_j}^{m-1}\pa_p q_j \Big|_{p_1}^{p_0}d\M'\quad\\
&&-\int_\M \left(\mu_{q_j} \nabla_h q_j \cdot \nabla_h q_{j,k_j}^{m-1}+ \nu_{q_j} \left(\frac{gp}{R\bar T}\right)^2
\pa_p q_j  \pa_p q_{j,k_j}^{m-1}\right)d\M\,.\nonumber
\eeq
Since, by definition, $k_j$ is larger than the given boundary functions, we obtain
\beq
&&\textnormal{on} \ \Gamma_\ll: \quad q_{j,k_j}^{m-1}\pa_n q_j = ((q_{j}-k_j)^+)^{m-1}\alpha_{\ll j}(q_{b\ll j}-q_j)\leq 0\,, \\
&&\textnormal{on} \ \Gamma_0: \quad q_{j,k_j}^{m-1}\pa_p q_j = ((q_{j}-k_j)^+)^{m-1}\alpha_{0j}(q_{b0j}-q_j)\leq 0\,.
\eeq
Moreover since $m\geq 2$ we can reformulate
\beq
\int_\M  \nabla_h q_j \cdot \nabla_h q_{j,k_j}^{m-1}d\M= (m-1)\int_\M q_{j,k_j}^{m-2}|\nabla_h q_{j,k_j}|^2 d\M=
\frac{4(m-1)}{m^2} \|\nabla_h q^{\frac{m}{2}}_{j,k_j}\|^2
\eeq
with an analogous calculation for the $p$-derivatives.

We now start with the derivation of the uniform boundedness of $q_c$ by testing equation \eqref{eq.qc} with $mq_{c,k_c}^{m-1}$ for any $m\geq2$. Using the preceding computations and the Young inequality, we obtain directly from the uniform boundedness of $q_v$ derived before:
 \beq
&&\frac{d}{dt}\int_\M q_{c,k_c}^m d\M \leq - 4 \frac{m-1}{m}(\mu_{q_c}\|\nabla_hq_{c,k_c}^\frac{m}{2}\|^2+ \nu_{q_c}\|\pa_pq_{c,k_c}^\frac{m}{2}\|_w^2)\nonumber\\
&&\qquad +m\int_\M \big(C_{cd}(q_v-q_{vs})q_c+C_{cn}(q_v-q_{vs})^+ -C_{ac}(q_c-q_{ac}^*)^+-C_{cr}q_cq_r\big)q_{c,k_c}^{m-1}d\M\nonumber\\
&&\quad \leq - 4 \frac{m-1}{m}(\mu_{q_c}\|\nabla_hq_{c,k_c}^\frac{m}{2}\|^2+ \nu_{q_c}\|\pa_pq_{c,k_c}^\frac{m}{2}\|_w^2)+
 C m \int_\M ( q_{c,k_c}^m +q_{c,k_c}^{m-1} )d\M\nonumber\\
 &&\quad\leq C m\int_\M ( 1+q_{c,k_c}^m)d\M\leq Cm(1+\|q_{c,k_c}\|_{L^m(\mathcal M)}^m),
\eeq
from which, by the Gronwall inequality and  the definition of $k_c$, one obtains
\begin{eqnarray}
  \|q_{c,k_c}\|_{L^m(\mathcal M)}^m(t)\leq e^{Cmt}(Cmt+\|q_{c,k_c}(0)\|_{L^m(\mathcal M)}^m)=e^{Cmt}Cm(\mathcal T+1),
\end{eqnarray}
for any $t\in(0,\mathcal T)$. Thanks to this estimate, we have
\beq
\|q_{c,k_c}\|_{L^m(\mathcal M)}(t)\leq e^{Ct}(C(\mathcal T+1))^{\frac1m}m^{\frac1m},
\eeq
from which, by taking $m\rightarrow\infty$, one gets
$\|q_{c,k_c}\|_{L^\infty(\mathcal M)}(t)\leq e^{Ct}$ for all $t\in [0,\mathcal T]$, and thus
\beq
\|q_c\|_{L^\infty(\mathcal M\times(0,\mathcal T))}\leq k_c+e^{C\mathcal T}=:q_c^*.
\eeq

We next apply the same method to derive also uniform boundedness of $q_r$ by employing the test function $m q_{r,k_r}^{m-1}$. The main difference to the previous estimates constitutes the additional vertical transport term of $q_r$, which we shall bound as follows:
\begin{align}
-Vm \int_\M q_{r,k_r}^{m-1}\pa_p&\left(\frac{p}{\bar T} q_r\right)d\M =-Vm\int_\mathcal M q_{r,k_r}^{m-1}\left[q_r\partial_p\left(\frac{p}{\bar T}\right)+\frac{p}{\bar T}\partial_pq_r\right]d\mathcal M\nonumber\\
=& - Vm \int_\M q_{r,k_r}^{m-1}(q_{r,k_r}+k_r)\pa_p\left(\frac{p}{\bar T}\right)d\M  -2V \int_\M\frac{p}{\bar T} q_{r,k_r}^{\frac{m}{2}}\pa_pq_{r,k_r}^{\frac{m}{2}}d\M\nonumber\\
\leq& C m \int_\M \left(1+q_{r,k_r}^m\right) d\M + \nu_{q_r}\frac{2(m-1)}{m}\|\pa_pq_{r,k_r}^{\frac{m}{2}}\|_w^2\,,
\end{align}
where we applied the Cauchy-Schwarz and Young inequalities.
Then the estimate for $q_r$ becomes
 \beq
&&\frac{d}{dt}\int_\M q_{r,k_r}^m d\M \leq - 4 \frac{m-1}{m}\big(\mu_{q_r}\|\nabla_hq_{r,k_r}^\frac{m}{2}\|^2+ \frac{\nu_{q_r}}{2}\|\pa_pq_{r,k_r}^\frac{m}{2}\|_w^2\big)+ Cm\int_\M (1+q_{r,k_r}^m)d\M \nonumber\\
&&\qquad +m\int_\mathcal M\big( C_{ac}(q_c-q_{ac}^*)^++C_{cr}q_cq_r-C_{ev}T(q_r^+)^\beta (q_{vs}-q_v)^+ \big)q_{r,k_r}^{m-1}d\M\nonumber\\
&&\quad \leq - 2 \frac{m-1}{m}(\mu_{q_r}\|\nabla_hq_{r,k_r}^\frac{m}{2}\|^2+ \nu_{q_r}\|\pa_pq_{r,k_r}^\frac{m}{2}\|_w^2) +  Cm\int_\M (1+q_{r,k_r}^m)d\M\,,
\eeq
where we have used the nonnegativity of $T$ and the uniform boundedness of $q_c$. We thus obtain
\beq
\frac{d}{dt}\int_\M q_{r,k_r}^m d\M\leq Cm(1+\|q_{r,k_r}\|_{L^m(\mathcal M)}^m).
\eeq
By the same argument as that for $q_c$ above, we get
 the following estimate for $q_r$:
\beq
\|q_r\|_{L^\infty(\mathcal M\times(0,\mathcal T))}\leq k_r+e^{C\mathcal T}=:q_r^*.
\eeq

(iv) \underline{Boundedness of $\theta$.} We finally derive a similar estimate for $\theta$. As before we set
\beq
\theta_{k_\theta}=(\theta-k_\theta)^+ \quad \textnormal{with}\quad  k_\theta=\sup_{t\in[0,\T]}(\|\theta_{b\ll}\|_{L^\infty(\Gamma_\ll)} +\|\theta_{b0}\|_{L^\infty(\pa \M')}+\|\theta_0\|_{L^\infty(\mathcal M)}),
\eeq
where $\theta_{b\ll}=\big(\frac{p_0}{p}\big)^{R/c_p}T_{b\ll}$ and  $\theta_{b0}=\big(\frac{p_0}{p}\big)^{R/c_p}T_{b0}$ accordingly.
Similar to above, one can deduce by integration by parts and using the
boundary condition (\ref{bound.ll}) that
\begin{equation}
  \label{06201}
  \int_\mathcal M\theta_{k_\theta}^{m-1}\Delta_h\theta d\mathcal M\leq -\frac{4(m-1)}{m}\|\nabla_h\theta_{k_\theta}^{\frac m2}\|^2\leq0.
\end{equation}
While for the diffusion term in the $p$-direction, the calculations are more involved. Recalling that $\theta=T\left(\frac{p_0}{p}\right)^{\frac{R}{c_p}}$, it follows from
the boundary conditions (\ref{bound.0})--(\ref{ass.qr}) that
\begin{eqnarray}
  \left(\frac{p_0}{p}\right)^{\frac{R}{c_p}}\theta_{k_\theta}^{m-1}\partial_p \left(\left(\frac{p}{p_0}\right)^{\frac{R}{c_p}}\theta\right)&=& \left(\frac{p_0}{p}\right)^{\frac{R}{c_p}}\theta_{k_\theta}^{m-1}\partial_pT =\left(\frac{p_0}{p}\right)^{\frac{R}{c_p}}\theta_{k_\theta}^{m-1}\alpha_{0T} (T_{b0}-T)\nonumber\\
  &=&\theta_{k_\theta}^{m-1}\alpha_{0T}(\theta_{b0}-\theta)\leq0, \qquad \mbox{on }\Gamma_0,
\end{eqnarray}
and
\beq
\left(\frac{p_0}{p}\right)^{\frac{R}{c_p}}\theta_{k_\theta}^{m-1}\partial_p \left(\left(\frac{p}{p_0}\right)^{\frac{R}{c_p}}\theta\right)= \left(\frac{p_0}{p}\right)^{\frac{R}{c_p}}\theta_{k_\theta}^{m-1}\partial_pT =0,\quad\mbox{on }\Gamma_1.
\eeq
Thus, due to integration by parts,
\begin{eqnarray}
  &&\int_\mathcal M \left(\frac{p_0}{p}\right)^{\frac{R}{c_p}}\partial_p\left(\left(\frac{gp} {R\bar T}\right)^2\partial_p\left(\left(\frac{p}{p_0}\right)^{\frac{R}{c_p}} \theta\right)\right) \theta_{k_\theta}^{m-1}d\mathcal M\nonumber\\
  &=&-\int_\mathcal M\left(\frac{gp}{R\bar T}\right)^2 \partial_p\left(\left(\frac{p}{p_0}\right)^{\frac{R}{c_p}} \theta\right) \partial_p\left(\left(\frac{p_0}{p}\right)^{\frac{R}{c_p}} \theta_{k_\theta}^{m-1}\right)d\mathcal M\nonumber\\
  &&+\int_{\mathcal M'} \left(\frac{p_0}{p}\right)^{\frac{R}{c_p}}\theta_{k_\theta}^{m-1} \left(\frac{gp}{R\bar T}\right)^2\partial_p \left(\left(\frac{p}{p_0}\right)^{\frac{R}{c_p}}\theta\right)d\mathcal M'\bigg|_{p_1}^{p_0}\nonumber\\
  &\leq& -\int_\mathcal M\left(\frac{gp}{R\bar T}\right)^2 \partial_p\left(\left(\frac{p}{p_0}\right)^{\frac{R}{c_p}} \theta\right) \partial_p\left(\left(\frac{p_0}{p}\right)^{\frac{R}{c_p}} \theta_{k_\theta}^{m-1}\right)d\mathcal M.
\end{eqnarray}
Direct calculations yield
\begin{eqnarray}
  &&\partial_p\left(\left(\frac{p}{p_0}\right)^{\frac{R}{c_p}} \theta\right) \partial_p\left(\left(\frac{p_0}{p}\right)^{\frac{R}{c_p}} \theta_{k_\theta}^{m-1}\right)\nonumber\\
  &=&(m-1)\theta_{k_\theta}^{m-2}(\partial_p\theta_{k_\theta})^2+\frac{R}{c_pp} \left((m-1)\theta_{k_\theta}^{m-2}\theta-\theta_{k_\theta}^{m-1}\right) \partial_p\theta_{k_\theta}-\left(\frac{R}{c_pp}\right)^2\theta \theta_{k_\theta}^{m-1}\nonumber\\
  &=&(m-1)\theta_{k_\theta}^{m-2}\left[ (\partial_p\theta_{k_\theta})^2+\frac{R}{c_pp}\left(\theta-\frac{\theta_{ k_\theta}}{m-1}\right)\partial_p\theta_{k_\theta}\right] -\left(\frac{R}{c_pp}\right)^2\theta \theta_{k_\theta}^{m-1}\nonumber\\
  &=&(m-1)\theta_{k_\theta}^{m-2} \left(\partial_p\theta_{k_\theta} +\frac{R}{2c_pp}\left(\theta-\frac{\theta_{ k_\theta}}{m-1}\right)\right)^2 \nonumber\\
  &&-\left(\frac{R}{c_pp}\right)^2\left[\frac{m-1}{4} \left(\theta-\frac{\theta_{k_\theta}}{m-1} \right)^2+\theta\theta_{k_\theta}\right]\theta_{k_\theta}^{m-2}.
\end{eqnarray}
Plugging this relation into the previous inequality,  one obtains by the Young inequality
\begin{eqnarray}
  &&\int_\mathcal M \left(\frac{p_0}{p}\right)^{\frac{R}{c_p}}\partial_p\left(\left(\frac{gp} {R\bar T}\right)^2\partial_p\left(\left(\frac{p}{p_0}\right)^{\frac{R}{c_p}} \theta\right)\right) \theta_{k_\theta}^{m-1}d\mathcal M\nonumber\\
  &\leq&\int_\mathcal M \left(\frac{gp}{R\bar T}\right)^2 \left(\frac{R}{c_pp}\right)^2\left[\frac{m-1}{4} \left(\theta-\frac{\theta_{k_\theta}}{m-1} \right)^2+\theta\theta_{k_\theta}\right]\theta_{k_\theta}^{m-2}d\mathcal M\nonumber\\
  &=&\left(\frac{g}{c_p\bar T}\right)^2\int_\mathcal M\left.\left[\frac{m-1}{4} \left(\theta-\frac{\theta_{k_\theta}}{m-1} \right)^2+\theta\theta_{k_\theta}\right]\right|_{\theta\geq k_\theta}\theta_{k_\theta}^{m-2}d\mathcal M\nonumber\\
  &=&\left(\frac{g}{c_p\bar T}\right)^2\int_\mathcal M\left.\left[\frac{m-1}{4} \left(\frac{m-2}{m-1}\theta_{k_\theta}+k_\theta \right)^2+\theta_{k_\theta}^2+k_\theta \theta_{k_\theta}\right]\right|_{\theta\geq k_\theta}\theta_{k_\theta}^{m-2}d\mathcal M\quad \nonumber\\
  &\leq&C\int_\mathcal M[m(\theta_{k_\theta}+k_\theta)^2+\theta_{k_\theta}^2+k_\theta \theta_{k_\theta}]\theta_{k_\theta}^{m-2} d\mathcal M \leq Cm\int_\mathcal M(1+\theta_{k_\theta}^m)d\mathcal M.
\end{eqnarray}
Combing the above estimate with (\ref{06201}) yields
\beq\label{bound.diff.theta}
\int_\mathcal M\tilde{\mathcal D}^\theta \theta \theta_{k_\theta}^{m-1} d\mathcal M\leq C m\int_\mathcal M(1+\theta_{k_\theta}^m)d\mathcal M.
\eeq
We now test equation (\ref{eq.theta}) with $m\theta_{k_\theta}^{m-1}$ and integrate by parts
 \beq
&&\frac{d}{dt}\int_\M \theta_{k_\theta}^m d\M \leq Cm\int_\M\left(\theta_{k_\theta}^m + 1\right)d\M\nonumber\\
&&\qquad +m\frac{L}{c_p}\int_\M \left(\frac{p_0}{p}\right)^\frac{R}{c_p} \big(C_{cd}(q_v-q_{vs})q_c + C_{cn}(q_v-q_{vs})^+\big)\theta_{k_\theta}^{m-1}d\M\nonumber\\
&&\qquad -m\frac{L}{c_p}\int_\M \left(\frac{p_0}{p}\right)^\frac{R}{c_p}C_{ev}T(q_r^+)^\beta(q_{vs}-q_v)^+ \theta_{k_\theta}^{m-1}d\M\nonumber\\
&& \quad \leq   Cm \int_\M (\theta_{k_\theta}^{m}+
1) d\M \,,
\eeq
where we made use of \eqref{bound.diff.theta},  the boundedness of $q_v, q_c$ and the nonnegativity of $T$. By the same argument as before for $q_c$ and $q_r$, it follows that
\beq
\|\theta\|_{L^\infty(\mathcal M\times(0,\mathcal T))}\leq k_\theta+e^{C\mathcal T},
\eeq
proving the upper bound for $T$, on the interval $[0,\mathcal T]$.
\end{proof}

\section{Global existence and well-posedness}\label{sec.uni}
In this section we prove our main result, i.e., the global existence and well-posedness, i.e., the  uniqueness and the continuous dependence with respect to the initial data,  of the
strong solutions to system (\ref{eq.qv})--(\ref{eq.T}), or equivalently
system (\ref{eq.qv})--(\ref{eq.qr}) with (\ref{eq.theta}), subject to the
boundary conditions (\ref{bound.0})--(\ref{bound.ll}).

The uniqueness, and the Lipschitz continuity of the solutions on the initial data,  is guaranteed by the following proposition, which, as we already mentioned in the Introduction,
requires the Lipschitz continuity of the saturation mixing stated in \eqref{LC}.

\begin{prop}\label{prop.uni}
Let $(T_i,q_{vi},q_{ci},q_{ri})$ for $i\in \{1,2\}$ be two strong solutions, on the interval $[0, \T]$, of  \eqref{eq.qv}--\eqref{eq.T}, subject to (\ref{bound.0})--(\ref{bound.ll}), with initial data $(T_i^0,q_{vi}^0,q_{ci}^0,q_{ri}^0)\in (L^\infty(\M))^4\cap (H^1(\mathcal M))^4$. Then, the following estimate holds
\beqs
\sup_{t\in[0,\T]}\Big(\|T_1-T_2\|^2+ \sum_{j\in\{v,c,r\}}\|q_{j1}-q_{j2}\|^2\Big)\leq C e^{C_0\T}\Big(\|T^0_1-T^0_2\|^2+ \sum_{j\in\{v,c,r\}}\|q^0_{j1}-q^0_{j2}\|^2\Big)
\eeqs
for a positive constant $C_0$, implying in particular the uniqueness of the solutions.
\end{prop}
\begin{proof}
The main difficulty in the proof of  the uniqueness and the continuous dependence on the initial data of the solutions is caused by the evaporation term $S_{ev}$ if the exponent $\beta \in (0,1)$. This problem can be circumvented by introducing the following new unknowns
\beq
Q=q_v+q_r, \qquad H=T-\frac{L}{c_p}(q_c+q_r)\,.\eeq
These quantities resemble the ones used in Hernandez-Duenas et al.\,\cite{HMSS}. However, in \cite{HMSS} the cloud water was not taken into account. In particular, $c_p H$ corresponds to the liquid water enthalpy, see e.g. Emanuel \cite{E}.
In the following we prove uniqueness by deriving typical $L^2$-estimates for the differences of the solutions in terms of estimates for the quantities
\beq\label{quant}
Q, q_c, q_r, H\,.
\eeq

We start with the estimate for $Q$, whose evolution is governed by the equation:
\begin{align}
\pa_t Q+\textbf{v}_h\cdot\nabla_h Q  +\omega \pa_p Q=& -V\pa_p\left(\frac{p}{\bar{T}}q_r\right)+S_{ac}+S_{cr}-S_{cd}+\mu_{q_v}\Delta_h Q +
(\mu_{q_r}-\mu_{q_v})\Delta_h q_r\nonumber\\
\label{Q}&+ \nu_{q_v}\pa_p\Big(\Big(\frac{gp}{R\bar{T}}\Big)^2\pa_p Q\Big)+
(\nu_{q_r}-\nu_{q_v})\pa_p\Big(\Big(\frac{gp}{R\bar{T}}\Big)^2\pa_p q_r\Big)\,.
\end{align}
%Then the equation for the difference $(Q_1-Q_2)$ reads as
%\begin{eqnarray}
%  &&\partial_t(Q_1-Q_2)+\textbf{v}_h\cdot\nabla_h (Q_1-Q_2)  +\omega \pa_p (Q_1-Q_2)-\mu_{q_v}\Delta_h(Q_1-Q_2)- \nu_{q_v}\partial_p\left(\left(\frac{gp}{R\bar T}\right)^2\partial_p(Q_1-Q_2)\right)\\
%  &=&(\mu_{q_r}-\mu_{q_v})\Delta_h(q_r^1-q_r^2)+(\nu_{q_r}-\nu_{q_v})\partial_p \left(\left(\frac{gp}{R\bar T}\right)^2\partial_p(q_r^1-q_r^2)\right)-V\partial_p\left(\frac{p}{\bar T}(q_r^1-q_r^2)\right)\\
%  &&+S_{ac}^1-S_{ac}^2+S_{cr}^1-S_{cr}^2-S_{cd}^1+S_{cd}^2.
%\end{eqnarray}
Let $Q_1,Q_2$ be two solutions of \eqref{Q}. We multiply the equation for the difference
$(Q_1-Q_2)$ by $(Q_1-Q_2)$, and integrate over $\mathcal M$.
Recalling \eqref{inc.v}--(\ref{noflux.v}), it follows from integration by parts and using the relevant boundary conditions, that
\begin{align}
\int_\mathcal M(\partial_t(Q_1-Q_2)+\textbf{v}_h\cdot\nabla_h (Q_1-Q_2)  +\omega \pa_p (Q_1-Q_2))(Q_1-Q_2) d\mathcal M =\frac12\frac{d}{dt}\|Q_1-Q_2\|^2.
\end{align}
By the boundary conditions (\ref{bound.0})--(\ref{bound.ll}), one can check that
\begin{eqnarray}
  \label{bound.1}&\mbox{on }\Gamma_{\ll}: &\partial_n(Q_1-Q_2)+\alpha_{\ll v}(Q_1-Q_2)=(\alpha_{\ll0}-\alpha_{\ll r})(q_{r1}-q_{r2}), \\
  &\mbox{on }\Gamma_0: &\partial_p(Q_1-Q_2)+\alpha_{0v}(Q_1-Q_2)=(\alpha_{0v}- \alpha_{0r})(q_{r1}-q_{r2}), \\
  &\mbox{on }\Gamma_1: &\partial_p(Q_1-Q_2)=\partial_p(q_{r1}-q_{r2})=0,\\
  \label{bound.2}&\mbox{on }\Gamma_{\ll}: &\partial_n(q_{r1}-q_{r2})=\alpha_{\ll r}(q_{r2}-q_{r1}),\quad\mbox{on }\Gamma_0:\ \ \partial_p(q_{r1}-q_{r2})=\alpha_{0r}(q_{r2}-q_{r1}).\qquad
\end{eqnarray}
For shortening the expressions, we use hereafter the
notation
\beq &&\text{Diff}\,(Q,q_r) := \\
&&\quad \mu_{q_v}\Delta_h Q +
(\mu_{q_r}-\mu_{q_v})\Delta_h q_r+ \nu_{q_v}\pa_p\Big(\Big(\frac{gp}{R\bar{T}}\Big)^2\pa_p Q\Big)+
(\nu_{q_r}-\nu_{q_v})\pa_p\Big(\Big(\frac{gp}{R\bar{T}}\Big)^2\pa_p q_r\Big)\nonumber
\eeq
for the diffusion terms in equation \eqref{Q} for $Q$.
It then follows from integration by parts, using the boundary conditions \eqref{bound.1}-\eqref{bound.2}, that
\beq
&&\int_\mathcal M(\text{Diff}(Q_1,q_{r1})-\text{Diff}(Q_2,q_{r2}))(Q_1-Q_2)d\mathcal M\nonumber\\
&=&\mu_{q_v}\int_{ \Gamma_\ll} (Q_1-Q_2)\pa_n (Q_1-Q_2) d\Gamma_\ll - \mu_{q_v}\|\nabla_h (Q_1-Q_2)\|^2\nonumber\\
&&+\nu_{q_v}\int_{\M'}\left(\frac{g p}{R\bar T}\right)^2\pa_p (Q_1-Q_2) (Q_1-Q_2) d\M' \bigg|_{p_1}^{p_0} - \nu_{q_v}\|\pa_p(Q_1-Q_2)\|_w^2\nonumber \\
&&+(\mu_{q_r}-\mu_{q_v})\int_{\Gamma_\ll}\partial_n (q_{r1}-q_{r2})(Q_1-Q_2)d\Gamma_\ll  \nonumber\\
&&+(\nu_{q_r}-\nu_{q_v})\int_{\M'}\left(\frac{g p}{R\bar T}\right)^2 \partial_p(q_{r1}-q_{r2})(Q_1-Q_2)d\M'\bigg|_{p_1}^{p_0}\nonumber\\
&&-(\mu_{q_r}-\mu_{q_v})\int_{\M}
\nabla_h(q_{r1}-q_{r2})\cdot\nabla_h(Q_1-Q_2)d\M \nonumber\\
&&- (\nu_{q_r}-\nu_{q_v}) \int_\M\left(\frac{gp}{R\bar T}\right)^2\pa_p(q_{r1}-q_{r2})\pa_p (Q_1-Q_2)d\M,
\eeq
%&&\leq - \frac{\mu_{q_v}}{2}\|\nabla_h (Q_1-Q_2)\|^2 + \frac{(\mu_{q_r}-\mu_{q_v})^2}{2\mu_{q_v}}\|\nabla_h (q_{r1}-q_{r2})\|^2 - \frac{\nu_{q_v}}{2}\|\pa_p (Q_1-Q_2)\|_w^2 \\
%&&\quad+ \frac{(\nu_{q_r}-\nu_{q_v})^2}{2\nu_{q_v}}\|\pa_p (q_{r1}-q_{r2})\|_w^2+ \mu_ {q_v}\int_{\Gamma_\ll}\alpha_{\ll v}(q_{v2}-q_{v1})(Q_1-Q_2) d\Gamma_\ll   \\
%&&\quad + \mu_ {q_r}\int_{\Gamma_\ll}\alpha_{\ll r}(q_{r2}-q_{r1})(Q_1-Q_2) d\Gamma_\ll +
%\nu_{q_v}\int_{\M'}\alpha_{0v}(q_{v2}-q_{v1})(Q_1-Q_2) d\M'  \\
%&&\quad +\nu_{q_r}\int_{\M'}\alpha_{0r}(q_{r2}-q_{r1}\big)(Q_1-Q_2)d\M' \\
Applying Young's inequality to the integrals over $\mathcal M$ and using the boundary conditions \eqref{bound.1}-\eqref{bound.2} to the
boundary integrals, one obtains, after some manipulations,
\beq
&&\int_\mathcal M(\text{Diff}(Q_1,q_{r1})-\text{Diff}(Q_2,q_{r2}))(Q_1-Q_2)d\mathcal M\nonumber\\
&\leq&- \frac{\mu_{q_v}}{2}\|\nabla_h (Q_1-Q_2)\|^2 + \frac{(\mu_{q_r}-\mu_{q_v})^2}{2\mu_{q_v}}\|\nabla_h (q_{r1}-q_{r2})\|^2 - \frac{\nu_{q_v}}{2}\|\pa_p (Q_1-Q_2)\|_w^2 \nonumber  \\
&&+ \frac{(\nu_{q_r}-\nu_{q_v})^2}{2\nu_{q_v}}\|\pa_p (q_{r1}-q_{r2})\|_w^2- \mu_ {q_v}\int_{\Gamma_\ll}\alpha_{\ll v}(Q_1-Q_2)^2  d\Gamma_\ll   \nonumber\\
&&+\int_{\Gamma_\ll}(\mu_ {q_r}\alpha_{\ll r} - \mu_ {q_v} \alpha_{\ll v})(q_{r2}-q_{r1})(Q_1-Q_2) d\Gamma_\ll \nonumber\\
&&-\nu_{q_v}\int_{\M'}\left(\frac{g p}{R\bar T}\right)^2 \alpha_{0v}(Q_1-Q_2)^2 d\M'\bigg|_{p=p_0} \nonumber\\
&&+\int_{\M'}(\nu_{q_r}\alpha_{0r}-\nu_{q_v}\alpha_{0v}) (q_{r2}-q_{r1}\big)(Q_1-Q_2)d\M'\bigg|_{p=p_0} \,.
\eeq
Since we do not want to make any restrictions on the boundary data
(i.e.,we also want to allow that, e.g., $\alpha_{0r}=0$, but
$\alpha_{0v}>0$), we apply Young's inequality with $\e>0$ sufficiently
small to the boundary integrals to finally estimate
them using the boundedness of the trace map  $H^1(\mathcal M)\mapsto
H^{\frac12}(\partial\mathcal M)\hookrightarrow L^2(\partial\mathcal M)$, see, e.g., \cite{D}, as follows
\beq
&&\int_{\Gamma_\ll}(\mu_ {q_r}\alpha_{\ll r} - \mu_ {q_v} \alpha_{\ll v})(q_{r2}-q_{r1})(Q_1-Q_2) d\Gamma_\ll \nonumber\\
&&+\int_{\M'}(\nu_{q_r}\alpha_{0r}-\nu_{q_v} \alpha_{0v})(q_{r2}-q_{r1}\big)(Q_1-Q_2)d\M' \bigg|_{p=p_0}\nonumber\\
&\leq&\e \int_{\pa\M}(Q_1-Q_2)^2d(\pa\M)+ C(\e)\int_{\pa\M}(q_{r1}-q_{r2})^2 d(\pa\M) \nonumber\\
&\leq &\frac{\mu_{q_v}}{4}\|\nabla_h (Q_1-Q_2)\|^2 + \frac{\nu_{q_v}}{4}\|\pa_p (Q_1-Q_2)\|_w^2 + C(\|Q_{1}-Q_{2}\|^2 + \|q_{r1}-q_{r2}\|^2)\nonumber\\
&&+ C_Q( \mu_{q_r}\|\nabla_h(q_{r1}-q_{r2})\|^2 + \nu_{q_r}\|\pa_p(q_{r1}-q_{r2})\|_w^2)\,.
\eeq
For the $L^2$-estimate we therefore obtain
\beq
&&\frac12\frac{d}{dt}\|Q_1-Q_2\|^2+\frac{\mu_{q_v}}{4} \|\nabla_h(Q_1-Q_2)\|^2 +\frac{\nu_{q_v}}{4} \|\pa_p(Q_1-Q_2)\|_w^2 \nonumber \\
&\leq&C_Q( \mu_{q_r}\|q_{r1}-q_{r2}\|^2 + \nu_{q_r}\|\pa_p(q_{r1}-q_{r2}\|_w^2) + C(\|Q_{1}-Q_{2}\|^2 + \|q_{r1}-q_{r2}\|^2)\nonumber\\
&&- V\int_\M\bigg((q_{r1}-q_{r2})\pa_p\left(\frac{p}{\overline{T}}\right) +\frac{p}{\overline{T}} \pa_p(q_{r1}-q_{r2})\bigg)(Q_1-Q_2) d\M\nonumber\\
%&&\quad -(\mu_{q_r}-\mu_{q_v})\int_\M \nabla_h (Q_1-Q_2)\cdot \nabla_h(q_{r1}-q_{r2})d\M\nonumber\\
%&&\quad -(\nu_{q_r}-\nu_{q_v})\int_\M \left(\frac{gp}{R\overline{T}}\right)^2 \pa_p(Q_1-Q_2)\pa_p (q_{r1}-q_{r2})d\M\nonumber\\
&&+ C_{ac}\int_\M
((q_{c1}-q_{ac}^*)^+-(q_{c2}-q_{ac}^*)^+)(Q_1-Q_2) d\M \nonumber\\
&&+C_{cr}\int_\M
\big((q_{c1}-q_{c2})q_{r1} +q_{c2}(q_{r1}-q_{r2})\big)(Q_1-Q_2) d\M\nonumber \\
&&-C_{cd}\int_\M ((q_{v1}-q_{v2})q_{c1}+q_{v2}(q_{c1}-q_{c2}))(Q_1-Q_2)d\M\nonumber\\
&& +C_{cd}\int_\M(q_{vs}(T_1)(q_{c1}-q_{c2})+q_{c2} (q_{vs}(p,T_1)-q_{vs}(p,T_2))(Q_1-Q_2) d\M\nonumber\\
&&- C_{cn}\int_\M
((q_{v1}-q_{vs}(p,T_1))^+-(q_{v2}-q_{vs}(p,T_2))^+)(Q_1-Q_2)  d\M \,.\label{4.13-li}
\eeq
The
Lipschitz continuity property of the saturation mixing ratio \eqref{LC} implies
\begin{eqnarray}
&&|(q_{v1}-q_{vs}(p,T_1))^+-(q_{v2}-q_{vs}(p,T_2))^+|\nonumber\\
&&\quad\leq |(q_{v1}-q_{vs}(p,T_1))^+-(q_{v2}-q_{vs}(p,T_1))^+| +|(q_{v2}-q_{vs}(p,T_1))^+-(q_{v2}-q_{vs}(p,T_2))^+|\nonumber\\
&&\quad\leq |q_{v1}-q_{v2}| + C|T_1-T_2|\,.\label{06202}
\end{eqnarray}
Using additionally the uniform boundedness of all moisture quantities as well as Young's inequality and rewriting $q_v$ and $T$ in terms
of the quantities in \eqref{quant}, we obtain:
\beq
&&\frac12\frac{d}{dt}\|Q_1-Q_2\|^2+ \frac{\mu_{q_v}}{4}\|\nabla_h (Q_1-Q_2)\|^2 +\frac{\nu_{q_{v}}}{4}\|\pa_p(Q_1-Q_2)\|_w^2\nonumber \\
&\leq& C \left(\|Q_1-Q_2\|^2+\|q_{r1}-q_{r2}\|^2+\|q_{c1}-q_{c2}\|^2+ \|H_1-H_2\|^2\right)\nonumber\\
\label{est.Q.uni}&& +C_Q(\mu_{q_r}\|\nabla_h (q_{r1}-q_{r2})\|^2 + \nu_{q_r} \|\pa_p (q_{r1}-q_{r2})\|_w^2).\qquad
\eeq

We next estimate the difference of the two rain water mixing ratios. We thereby bound the diffusion and the vertical transport term with terminal velocity $V$ as follows:
\beq
&&\int_\M(q_{r1}-q_{r2}) {\cal D}^{q_r}(q_{r1}-q_{r2})d\M -V \int_\M (q_{r1}-q_{r2}) \pa_p \left(\frac{p}{\bar{T}}(q_{r1}-q_{r2})\right)d\M\nonumber\\
&&\quad = -\mu_{q_r}\int_{\Gamma_\ll} \alpha_{\ll r}(q_{r1}-q_{r2})^2d\Gamma_\ll-\mu_{q_r}\|\nabla_h(q_{r_1}-q_{r_2})\|^2 \nonumber\\
&&\qquad -\nu_{q_r}\int_{\M'}\left(\frac{gp}{R\bar{T}}\right)^2 \alpha_{0r}(q_{r1}-q_{r2})^2d\M'\Big|_{p=p_0}-\nu_{q_r}\|\pa_p(q_{r1}-q_{r2})\|_w^2\nonumber\\
&&\qquad - V\int_\M \bigg[(q_{r1}-q_{r2})^2 \pa_p \left(\frac{p}{\overline{T}}\right) +\frac{p}{\overline{T}}(q_{r1}-q_{r2})\pa_p(q_{r1}-q_{r2})\bigg]d\M \nonumber\\
&&\quad \leq  -\mu_{q_r}\|\nabla_h(q_{r_1}-q_{r_2})\|^2 -\frac{\nu_{q_r}}{2}\|\pa_p(q_{r1}-q_{r2})\|_w^2 + C\|q_{r_1}-q_{r2}\|^2\,,
\eeq
and we get  for the $L^2$-estimate of $q_{r1}-q_{r2}$:
\beq
&&\frac12\frac{d}{dt}\|q_{r 1}-q_{r 2}\|^2+\mu_{q_r}\|\nabla_h (q_{r1}-q_{r2})\|^2 +\frac{\nu_{q_r}}{2}\|\pa_p(q_{r1}-q_{r2})\|_w^2 \nonumber\\
&\leq&C\|q_{r_1}-q_{r2}\|^2
+ C_{ac}\int_\M
((q_{c1}-q_{ac}^*)^+-(q_{c2}-q_{ac}^*)^+)(q_{r1}-q_{r2}) d\M \nonumber\\
&& +C_{cr}\int_\M\left[
(q_{c1}-q_{c2})q_{r1}(q_{r1}-q_{r2}) +q_{c2}(q_{r1}-q_{r2})^2\right] d\M \nonumber\\
&& -C_{ev} \int_\M (T_1-T_2)q_{r1}^\beta(q_{vs}(p,T_1)-q_{v1})^+ (q_{r1}-q_{r2}) d\M\nonumber\\
&&-C_{ev} \int_\M T_2(q_{r1}^\beta-q_{r2}^\beta)(q_{vs}(p,T_1)-q_{v1})^+ (q_{r1}-q_{r2})d\M\nonumber\\
&&-C_{ev} \int_\M T_2 q_{r2}^\beta((q_{vs}(p,T_1)-q_{v1})^+-(q_{vs}(p,T_2)
-q_{v2})^+)(q_{r1}-q_{r2}) d\M\,.\label{4.17-li}
\eeq
We can now use the monotonicity property in the evaporation term $(q_{r1}^\beta-q_{r2}^\beta)(q_{r1}-q_{r2})\geq 0$ and recall (\ref{06202}) to estimate further replacing $T$ in terms of $H,q_c$ and $q_r$
\beq
&&\frac12\frac{d}{dt}\|q_{r 1}-q_{r 2}\|^2 +  \mu_{q_r}\|\nabla_h (q_{r1}-q_{r2})\|^2 + \frac{\nu_{q_r}}{2}\|\pa_p(q_{r1}-q_{r2})\|_w^2\nonumber \\
&&\label{est.qr.uni}\quad \leq C \left(\|Q_1-Q_2\|^2+\|q_{r1}-q_{r2}\|^2+\|q_{c1}-q_{c2}\|^2+ \|H_1-H_2\|^2\right).
\eeq

We now turn to the difference of the cloud water mixing ratios. We estimate the boundary terms arising from partial integration of the diffusion operator as above, using (\ref{06202}) and the Lipschitz continuity of $q_{vs}$, (\ref{LC}), and obtain
\beq
&&\frac12\frac{d}{dt}\|q_{c1}-q_{c2}\|^2 + \mu_{q_c}\|\nabla_h (q_{c1}-q_{c2})\|^2 + \nu_{q_c}\|\pa_p(q_{c1}-q_{c2})\|_w^2\nonumber\\
&\leq&  C_{cd} \int_\M (q_{v1}-q_{v2}-(q_{vs}(p,T_1)-q_{vs}(p,T_2))q_{c1}(q_{c1}-q_{c2})d\M \nonumber\\
&&+C_{cd} \int_\M (q_{v2} -q_{vs}(p,T_2))(q_{c1}-q_{c2}) (q_{c1}-q_{c2})d\M \nonumber\\
&& +C_{cn}\left((q_{v1}-q_{vs}(p,T_1))^+-(q_{v2}-q_{vs}(p,T_2))^+\right) (q_{c1}-q_{c2}) d\M\nonumber\\
&& - C_{ac}\int_\M
((q_{c1}-q_{ac}^*)^+-(q_{c2}-q_{ac}^*)^+)(q_{c1}-q_{c2}) d\M\nonumber\\
&& - C_{cr}\int_\M
\left[(q_{c1}-q_{c2})^2q_{r1}+q_{c2}(q_{r1}-q_{r2})(q_{c1}-q_{c2})\right] d\M\nonumber \\
\label{est.qc.uni}
&\leq& C \left(\|Q_1-Q_2\|^2+\|q_{r1}-q_{r2}\|^2+\|q_{c1}-q_{c2}\|^2+\|H_1-H_2\|^2\right).
\eeq

To close the estimates it remains to bound $(H_1-H_2)$. The equation for $H$ reads
\beq
&&\pa_t H+\textbf{v}_h\cdot\nabla_h H+\omega \pa_p H=\nonumber\\
&&=\frac{RT}{c_pp}\omega -\frac{L}{c_p}V\pa_p\Big(\frac{p}{\bar T}
q_r\Big)+\mu_T \Delta_h H + \nu_T \pa_p\Big(\Big(\frac{g p}{R\bar{T}}\Big)^2\pa_p H\Big)\nonumber \\
&&\quad -\frac{L}{c_p}\left((\mu_{q_c}-\mu_T)\Delta_h q_{c}+(\nu_{q_c}-\nu_T)\pa_p\Big(\Big(\frac{g p}{R\bar{T}}\Big)^2\pa_p q_c\Big)\right)\nonumber \\
&&\quad  -\frac{L}{c_p}\left(
(\mu_{q_r}-\mu_T)\Delta_h q_r+(\nu_{q_r}-\nu_T)\pa_p\Big(\Big(\frac{g p}{R\bar{T}}\Big)^2\pa_p q_r\Big)\right)\,.
\eeq
As it can be seen here from the thermodynamic equation the antidissipative term containing the vertical velocity $\omega$ appears on the right hand side. In previous estimates we could circumvent estimations involving this term by switching to the potential temperature $\theta$. However, this is not possible here, since for an alternative definition of the form $\widetilde H=\theta - \frac{L}{c_p}\big(\frac{p_0}{p}\big)^\frac{R}{c_p}(q_c+q_r)$, which would also allow for a cancellation of the source terms from phase changes, the vertical velocity term would appear again from the vertical advection term due to the pressure function multiplying the moisture quantities. We therefore stick to the previously introduced quantity $H$.

Testing again the equation for the difference $(H_1-H_2)$ against $(H_1-H_2)$ and treating the boundary terms arising from the diffusion operators and the additional vertical transport term in a similar way as for $(Q_1-Q_2)$  above, we obtain:
\begin{align}
\frac12&\frac{d}{dt}\|H_1-H_2\|^2\leq -\mu_T\|\nabla_h(H_1-H_2)\|^2 -\nu_T\|\pa_p(H_1-H_2)\|_w^2 \nonumber\\
&+ C(\|q_{c1}-q_{c2}\|^2_{L^2(\pa \M)}+\|q_{r1}-q_{r2}\|^2_{L^2(\pa \M)}) + \e \|H_{1}-H_{2}\|^2_{L^2(\pa \M)}  \nonumber\\
&- \frac{L}{c_p}\int_\mathcal M \nabla_h(H_1-H_2)\cdot\big((\mu_{q_c}-\mu_T)\nabla_h(q_{c1}-q_{c2}) +
 (\mu_{q_r}-\mu_T)\nabla_h(q_{r1}-q_{r2})\big)d\M \nonumber\\
&- \frac{L}{c_p}\int_\mathcal M \left(\frac{g p}{R\bar T}\right)^2\pa_p(H_1-H_2)\cdot
\big((\nu_{q_c}-\nu_T)\pa_p(q_{c1}-q_{c2}) + (\nu_{q_r}-\nu_T)\pa_p(q_{r1}-q_{r2})\big)d\M\nonumber\\
&- \frac{L}{c_p}\int_\mathcal M\left[ V\frac{p}{\overline{T}} (q_{r1}-q_{r2})\pa_p(H_1-H_2) + \frac{R\omega}{c_p p } (T_1-T_2)(H_1-H_2)\right] d\M
\nonumber\\
\leq& -\frac{\mu_T}{2}\|\nabla_h(H_1-H_2)\|^2 -\frac{\nu_T}{2}\|\pa_p(H_1-H_2)\|_w^2 + C_H \mu_{q_c}\|\nabla_h (q_{c1}-q_{c2})\|^2\nonumber\\
&+C_H( \mu_{q_r} \|\nabla_h (q_{r1}-q_{r2})\|^2 + \nu_{q_c}\|\pa_p (q_{c1}-q_{c2})\|_w^2+\nu_{q_r}\|\pa_p (q_{r1}-q_{r2})\|_w^2)\nonumber\\
\label{est.H.uni}
&+ C \left(\|H_1-H_2\|^2+\|q_{r1}-q_{r2}\|^2+\|q_{c1}-q_{c2}\|^2\right)\,,
\end{align}
where $\varepsilon$ is chosen sufficiently small; we have also used here  the boundedness of the trace map $H^1(\mathcal M)\mapsto L^2(\partial\mathcal M)$ and Young's inequality. Moreover, since according to \eqref{reg.v} we have $\omega\in L^\infty(0,\T;L^2(\M))$, we have also used above the following bound for the integral containing the vertical velocity component $\omega$:
\begin{align}
&\left|\int_\M\frac{R\omega}{c_p p } (T_1-T_2)(H_1-H_2) d\M\right|\nonumber\\
\leq& C\Big(\|H_1-H_2\|_{L^4(\M)}^2+\sum_{j\in\{c,r\}} \|q_{j1}-q_{j2}\|_{L^4(\M)}^2\Big)\nonumber\\
\leq& C\bigg(\|H_1-H_2\|_{H^1(\mathcal M)}^{\frac32}\|H_1-H_2\|^{\frac12}+ \sum_{j\in\{c,r\}} \|q_{j1}-q_{j2}\|_{H^1(\M)}^{\frac32}\|q_{j1}-q_{j2}\|^{\frac12}\bigg)\nonumber\\
\leq& \varepsilon\bigg(\|\nabla(H_1-H_2)\|^2+\sum_{j\in\{c,r\}}  \|\nabla(q_{j1}-q_{j2})\|^2\bigg)\nonumber\\
&+C_\varepsilon\bigg(\|H_1-H_2\|^2+\sum_{j\in\{c,r\}}  \|q_{j1}-q_{j2}\|^2\bigg).
\end{align}

We now combine the estimates \eqref{est.Q.uni}, \eqref{est.qr.uni}, \eqref{est.qc.uni} and \eqref{est.H.uni} by introducing
\beq
J(t)=\frac12\big(A \|Q_1-Q_2\|^2 + \|q_{c1}-q_{c2}\|^2+\|q_{r1}-q_{r2}\|^2+ B \|H_1-H_2\|^2\big)\,,
\eeq
where $A$ and $B$ are positive constants  chosen accordingly (e.g., \,$A= \frac{1}{2C_Q}$ and $ B=\frac{1}{2C_H}$), such that for some $C_0>0$
\beq
\frac{d J }{dt}\leq C_0 J\,
\eeq
and we conclude the proof by applying the Gronwall inequality. \end{proof}

We  are now ready to prove our main result, Theorem \ref{thm}.

\begin{proof}[\textbf{Proof of Theorem \ref{thm}}] The uniqueness and continuous dependence on the initial data is
an immediate corollary of Proposition \ref{prop.uni}. Therefore, it remains to
prove the global existence. By Proposition \ref{prop.local}, under the assumptions in Theorem \ref{thm}, there is a unique local strong solution $(q_v, q_c, q_r, \theta)$, with
\beq
(q_v, q_c, q_r, \theta)\in C([0,\mathcal T_0]; H^1(\mathcal M))\cap L^2(0,\mathcal T_0; H^2(\mathcal M)).
\eeq
We extend the unique strong solution $(q_v, q_c, q_r, \theta)$ to the maximal
time of existence $\mathcal T_*$. If $\mathcal T_*=\infty$, then one obtains
a global strong solution. Suppose that $\mathcal T_*<\infty$, then one obviously has
\begin{equation}\label{5.30.0}
\lim_{\mathcal T\rightarrow\mathcal T_*^-}\|(q_v, q_c, q_r, \theta)\|_{L^\infty(0,\mathcal T; H^1(\mathcal M))\cap L^2(0,\mathcal T, H^2(\mathcal M))}=\infty.
\end{equation}

By Proposition \ref{prop.comp}, we have
$\|(q_v, q_c, q_r, \theta)\|_{L^\infty(\mathcal M\times(0,\mathcal T_*))}
\leq C$, for a positive constant $C$ that depends continuously on $\mathcal T_*$. Let $\varepsilon$ be a small enough
positive time to be specified later. We will estimate the norms to the solution in the time
interval $(\mathcal T_*-\varepsilon, \mathcal T)$, for $\mathcal T\in(\mathcal T_*-\varepsilon,\mathcal T_*)$.
For simplicity, we denote $\mathbf{U}=(q_v, q_c, q_r, \theta)$. Applying Corollary
\ref{corexistparabolic} in section \ref{secellipticparabolic}, below, to
system (\ref{eq.qv})--(\ref{eq.qr}) with (\ref{eq.theta}), and
recalling that $\|(q_v, q_c, q_r, \theta)\|_{L^\infty(\mathcal M\times(0,\mathcal T_*))}\leq C$, we have
\begin{eqnarray}
  &&\|\mathbf{U}\|_{L^\infty(\mathcal T_*-\varepsilon, \mathcal T; H^1(\mathcal M))\cap L^2(\mathcal T_*-\varepsilon, \mathcal T; H^2(\mathcal M))}\label{5.30.2}\\
  &\leq& C(1+\|\textbf{v}_h\cdot\nabla_h \mathbf{U}\|_{L^2(\mathcal M\times(\mathcal T_*-\varepsilon, \mathcal T))}+\|\omega \pa_p \mathbf{U}\|_{L^2(\mathcal M\times(\mathcal T_*-\varepsilon, \mathcal T))}+\|\partial_p \mathbf{U}\|_{L^2(\mathcal M\times(\mathcal T_*-\varepsilon, \mathcal T))}),\nonumber
\end{eqnarray}
for a positive constant $C$ independent of $\mathcal T<\mathcal T_*$.

The
same argument as for (\ref{5.30.1}) yields
\begin{eqnarray}
\|\textbf{v}_h\cdot\nabla_h  \mathbf{U}\|_{L^2(\mathcal M\times(\mathcal T_*-\varepsilon, \mathcal T))}&\leq& C(\mathcal T-\mathcal T_*+\varepsilon)^{\frac{\sigma-3}{2\sigma}-\frac1r}\|\textbf{v}_h\|_{L^r(\mathcal T_*-\varepsilon,\mathcal T; L^\sigma(\mathcal M))}\nonumber \\
&&\cdot\| \mathbf{U}\|_{L^\infty(\mathcal T_*-\varepsilon,\mathcal T; H^1(\mathcal M))\cap L^2(\mathcal T_*-\varepsilon,\mathcal T; H^2(\mathcal M))}\nonumber\\
&\leq&C\varepsilon^{\frac{\sigma-3}{2\sigma}-\frac1r}\| \mathbf{U}\|_{L^\infty(\mathcal T_*-\varepsilon,\mathcal T; H^1(\mathcal M))\cap L^2(\mathcal T_*-\varepsilon,\mathcal T; H^2(\mathcal M))},\qquad
\end{eqnarray}
for a positive constant $C$ independent of $\mathcal T<\mathcal T_*$, with an analogous estimate for $\omega\pa_p  \mathbf{U}$. By the
H\"older inequality one moreover gets
\begin{eqnarray}
\|\partial_p  \mathbf{U}\|_{L^2(\mathcal M\times(\mathcal T_*-\varepsilon, \mathcal T))}
&\leq& C(\mathcal T-\mathcal T_*+\varepsilon)^{\frac12}\| \mathbf{U}\|_{L^\infty(\mathcal T-\mathcal T_*+\varepsilon,\mathcal T; H^1(\mathcal M))}\nonumber\\
&\leq&C\varepsilon^{\frac12}\| \mathbf{U}\|_{L^\infty(\mathcal T-\mathcal T_*+\varepsilon,\mathcal T; H^1(\mathcal M))},
\end{eqnarray}
for a positive constant $C$ independent of $\mathcal T<\mathcal T_*$.

Plugging the above two estimates into (\ref{5.30.2}), and choosing $\varepsilon$ small enough, one concludes
\beq
\|\mathbf{U}\|_{L^\infty(\mathcal T_*-\varepsilon, \mathcal T; H^1(\mathcal M))\cap L^2(\mathcal T_*-\varepsilon, \mathcal T; H^2(\mathcal M))}\leq C,
\eeq
for a positive constant $C$ independent of $\mathcal T<\mathcal T_*$, which  contradicts (\ref{5.30.0}). Therefore, we must have $\mathcal T_*=\infty$,
in other words, the solution $(q_v, q_c, q_r, \theta)(t)$ can be extended to be a global
solution.
\end{proof}

\section{Elliptic and parabolic problems on cylindrical domains}
\label{secellipticparabolic}
As mentioned in the introduction, this technical section is devoted to establishing the existence and uniqueness of solutions (weak or strong) to
some elliptic or parabolic problems, subject to the Robin boundary
conditions on the cylindrical domains.
Since the results stated in this section hold for any finite dimensional
space, and in order to state and prove the results in the general settings,
the notations used in this section are independent of those used in
the previous sections.

Let $M\subseteq\mathbb R^N$ be a bounded domain with smooth boundary,
$L_0$ and $L_1$ two numbers, with $L_0<L_1$, and set the cylindrical domain
$\Omega=M\times(L_0, L_1)$. For a spatial variable $x\in\Omega$, we denote by
$x'\in M$ the first $N$ components of $x$, that is $x'=(x^1, \cdots, x^N)$,
and by $z$ the last component of $x$, that is $z=x^{N+1}$. We divide the
boundary into the lateral boundary $\Gamma_\ll$, the
upper boundary $\Gamma_1$ and the lower boundary $\Gamma_0$, that is $\partial\Omega=\Gamma_\ll\cup\Gamma_1\cup\Gamma_0$, where
\begin{eqnarray}
  \Gamma_\ll=\partial M\times[L_0, L_1], \quad\Gamma_0=M\times\{L_0\}, \quad\Gamma_1=M\times\{L_1\},
\end{eqnarray}
and set $\Gamma_{01}=\Gamma_0\cup\Gamma_1$. Note that the unit outward normal
vector field on the boundary $\partial\Omega$ is piecewise smooth. In view of this, we use different notations to distinguish
the normal vectors on different portions of the boundary: we
denote by $n$ the unit outward normal vector on $\Gamma_\ll$, while the unit outward normal vector
on $\Gamma_{01}$ is denoted by $\nu$.

\subsection{Linear elliptic problem}
In this subsection we establish the existence, uniqueness and regularity of weak solutions to  the following elliptic problem:
\begin{equation}
  \label{ELPROB}
  \left\{
  \begin{array}{lr}
  -\Delta_h u-\partial_z(a\partial_zu)+bu=f, &\mbox{in }\Omega,\\
  \partial_nu+\alpha u=\varphi,&\mbox{on }\Gamma_\ll,\\
  \partial_\nu u+\beta u=\psi,&\mbox{on }\Gamma_{01},
  \end{array}
  \right.
\end{equation}
where $a, b, \alpha, \beta, f, \varphi$ and $\psi$ are given functions. As above  $\nabla_h$ and $\Delta_h$ denote the gradient and Laplacian in the horizontal coordinates $(x^1,\dots,x^N)$.

Weak solutions to (\ref{ELPROB}) are defined as follows.

\begin{df}\label{defel}
A function $u\in H^1(\Omega)$ is called a weak solution to the elliptic
problem (\ref{ELPROB}), if the following equation holds
\begin{eqnarray}
  \int_\Omega( \nabla_hu\cdot \nabla_h\phi+a\partial_zu\partial_z\phi+bu\phi)dx
  +\int_{\Gamma_\ll}\alpha u\phi d\Gamma_\ll+\int_{\Gamma_{01}}a\beta u\phi dx\nonumber'\\
  =\int_\Omega f\phi dx+\int_{\Gamma_\ll}\varphi\phi d\Gamma_\ll+\int_{\Gamma_{01}}a\psi\phi dx'
\end{eqnarray}
holds for any $\phi\in H^1(\Omega)$. If $b\equiv\alpha\equiv\beta\equiv0$,
we ask for the additional constraint $\int_\Omega udx=0$ on $u$,
and assume that the following compatibility condition holds
\beq
\int_\Omega f dx+\int_{\Gamma_\ll}\varphi d\Gamma_\ll+\int_{\Gamma_{01}}a\psi dx'=0.
\eeq
\end{df}

Existence and uniqueness of weak solutions to (\ref{ELPROB}) is stated in the following proposition, which can be proven in the standard way
by the Lax-Milgram Lemma \cite{LM}.

\begin{prop}[Existence and uniqueness of weak solutions]\label{prop.weak.elliptic}
Assume that the functions $a, b, \alpha, \beta, f, \varphi$ and $\psi$ satisfy \begin{eqnarray}
  &a, b\in L^\infty(\Omega), \quad \alpha\in L^\infty(\Gamma_\ll),\quad\beta\in L^\infty(\Gamma_{01}), \quad
  \lambda\leq a\leq\Lambda, \quad 0\leq b, \alpha,\beta \leq\Lambda,\\
  &f\in L^2(\Omega),\quad\varphi\in L^2(\Gamma_\ll),\quad\psi\in L^2(\Gamma_{01}),
\end{eqnarray}
for some positive numbers $\lambda$ and $\Lambda$. Then, there is
a unique weak solution $u$ to (\ref{ELPROB}), satisfying
\beq
\|u\|_{H^1(\Omega)}\leq C(\|f\|_{L^2(\Omega)}+\|\varphi\|_{L^2(\Gamma_\ll)} +\|\psi\|_{L^2(\Gamma_{01})}),
\eeq
for a positive constant $C$ depending only on $\Omega$ and the coefficients $a, b, \alpha$ and $\beta$.
\end{prop}

We are going to study the $H^2$ regularity of the weak solutions established
in Proposition \ref{prop.weak.elliptic}. To this end, we start with the
corresponding result for the
elliptic problem of a special case stated in the next lemma.

\begin{lemma}[$H^2$ regularity: a special case]
\label{lem.reg.elliptic.special}
Assume that $a, f$ and $\varphi$ satisfy the assumptions in Proposition \ref{prop.weak.elliptic}, and
let $u\in H^1(\Omega)$ be a weak solution to
\begin{equation}\label{eq1li}
  \left\{
  \begin{array}{lr}
  - \Delta_h u-\partial_z( a\partial_z u)= f,&\mbox{in } \Omega, \\
  \partial_n u= \varphi,&\mbox{on } \Gamma_\ll, \\
  \partial_\nu u=0,&\mbox{on } \Gamma_{01}.
  \end{array}
  \right.
\end{equation}
Suppose in addition that
$a\in C(\overline\Omega), \partial_za\in L^{\infty}(\Omega)$ and
$\varphi\in H^{\frac12}(\Gamma_\ll).$

Then, $u\in H^{2}(\Omega)$, and the following estimate holds
\beq
\|u\|_{H^{2}(\Omega)}\leq C(\|f\|_{L^2(\Omega)}+\|\varphi\|_{H^{\frac12}(\Gamma_\ll)} +\|u\|_{H^{1}(\Omega)}),
\eeq
where $C$ is a positive constant depending only on $\Omega$,
the modulus of continuity of $a$, $\|a\|_{L^\infty(\Omega)}$ and $\|\partial_za\|_{L^\infty(\Omega)}$.
\end{lemma}

\begin{proof}
Our strategy to prove the conclusion is as follows: we extend the weak
solution appropriately to a larger domain, multiply it by some cutoff function, and
show that the resultant satisfies some elliptic equations subject to
the Neumann boundary conditions in a smooth
domain, for which the classical regularity and elliptic estimates apply.

We first extend the boundary function $\varphi$ to the
whole domain $\Omega$. By the trace theorem, one can find an extension  $\varphi_{ext}\in H^{1}(\Omega)$ of $\varphi$ to the domain $\Omega$, such
that $\varphi_{ext}=\varphi$ on $\Gamma_\ll$, in the sense of trace, and
\beq
c\|\varphi_{ext}\|_{H^{1}(\Omega)}\leq \|\varphi\|_{H^{\frac12}(\Gamma_\ll)}\leq C\|\varphi_{ext}\|_{H^{1}(\Omega)},
\eeq
for two positive constants $c$ and $C$ depending
only on $\Omega$. Thanks to this, in the rest of the proof we always suppose
that the boundary function $\varphi$ has already been extended to the whole
domain $\Omega$, such that $\varphi\in H^{1}(\Omega)$ and the above estimate holds with $\varphi_{ext}=\varphi$.

We first extend $u$, so that it satisfies
a similar elliptic equation in the extended domain. To this end, let $\tilde\Omega=M\times(2L_0-L_1, 2L_1-L_0)$, and extend $u$
in $z$ evenly, with respect to the plane $z=L_0$ and $z=L_1$,
that is we define
\beq
\tilde u(x',z)=
\left\{
\begin{array}{ll}
  u(x',2L_0-z), &z\in (2L_0-L_1, L_0),\\
  u(x',z),& z\in [L_0,L_1], \\
  u(x',2L_1-z),& z\in (L_1,2L_1-L_0),
\end{array}
\right.
\eeq
for $x=(x',z)\in\tilde \Omega$. The extended functions $\tilde a, \tilde f, \tilde\alpha, \tilde \varphi$ are defined in the similar way as
$\tilde u$. Then one can check that $\tilde u\in H^1(\tilde\Omega)$ is a weak solution to
\begin{equation}\label{eq1li.2}
  \left\{
  \begin{array}{lr}
  - \Delta_h\tilde u-\partial_z(\tilde a\partial_z\tilde u)=\tilde f,&\mbox{in }\tilde\Omega, \\
  \partial_n\tilde u=\tilde\varphi,&\mbox{on }\tilde\Gamma_\ll, \\
  \partial_\nu\tilde u=0,&\mbox{on }\tilde\Gamma_{01},
  \end{array}
  \right.
\end{equation}
where $\tilde\Gamma_\ll=\partial M\times(2L_0-L_1,2L_1-L_0)$ and $\tilde\Gamma_{01}=M\times\{2L_0-L_1,2L_1-L_0\}$.

We will introduce some appropriate truncation function of $\tilde u$,
and show that the truncation satisfies some elliptic equation, subject
to the Neumann boundary condition, in a smooth domain. To this end,
let us first take a smooth bounded domain $\mathcal O$, with
$\Omega\subseteq\mathcal O\subseteq\tilde\Omega$,
$\partial\mathcal O\cap\tilde\Omega\cap\Gamma_{01}=\emptyset$
and $\partial\mathcal O\cap\tilde\Omega\cap\tilde\Gamma_{01}=\emptyset$,
and choose a function $\eta\in C_0^\infty(\mathbb R^{N+1})$, with $\eta\equiv1$
on $\Omega$, and $\eta\equiv0$ on $\tilde\Omega\setminus\mathcal O$. Since
$\tilde u$ is a weak solution to (\ref{eq1li.2}), choosing $\phi\eta$ as a
testing function yields
\begin{eqnarray}
  \int_{\tilde\Omega}[ \nabla_h\tilde u\cdot \nabla_h(\phi\eta)+\tilde a\partial_z \tilde u\partial_z(\phi\eta)]dx=
  \int_{\tilde\Omega}\tilde f\phi\eta dx+\int_{\tilde\Gamma_\ll}\tilde\varphi \eta\phi d\Gamma_\ll, \label{eq2li}
\end{eqnarray}
for any $\phi\in H^1(\tilde\Omega)$. Noticing that $\partial\mathcal O=(\partial\mathcal O\cap\tilde\Omega)\cup(\partial\mathcal O\cap\tilde\Gamma_s)$, and $ \nabla_h\eta=0$ on $\partial\mathcal O\cap\tilde\Omega$,
integration by parts yields
\begin{eqnarray}
  \int_{\tilde\Omega} \nabla_h\tilde u \cdot \nabla_h(\phi\eta)dx  %&=&\int_{\mathcal O} \nabla_h\tilde u\cdot( \nabla_h \eta\phi+\eta \nabla_h\phi)dx \nonumber\\
%  &=&\int_{\mathcal O}[( \nabla_h(\tilde u\eta)-\tilde u \nabla_h\eta)\cdot \nabla_h\phi+ \nabla_h\tilde u\cdot \nabla_h\eta\phi ]dx\nonumber\\
  &=&\int_{\mathcal O}[ \nabla_h(\tilde u\eta)\cdot \nabla_h\phi+( \nabla_h\cdot(\tilde u \nabla_h\eta)+ \nabla_h \tilde u\cdot \nabla_h\eta)\phi]dx\nonumber\\
  &&-\int_{\partial\mathcal O\cap\tilde\Gamma_\ll}\tilde u\partial_n\eta\phi d\Gamma_\ll.\label{eq3li}
\end{eqnarray}
Noticing that $\partial_z\eta=0$ on $\partial\mathcal O\cap\tilde\Omega$, it follows from integration by parts that
\begin{eqnarray}
  &&\int_{\tilde\Omega}a\partial_z\tilde u\partial_z(\phi\eta)dx
%  =\int_{\mathcal O}a\partial_z\tilde u(\partial_z\phi \eta+\partial_z\eta\phi)dx\nonumber\\
%  &=&\int_{\mathcal O}(a\partial_z(\tilde u\eta)\partial_z\phi-a\partial_z\eta\tilde u\partial_z\phi+a\partial_z\tilde u\partial_z\eta\phi)dx\nonumber\\
  =\int_{\mathcal O}a\partial_z(\tilde u\eta)\partial_z\phi dx+\int_{\mathcal O}(\partial_z(a\partial_z\eta\tilde u)+a\partial_z\tilde u\partial_z\eta)\phi dx.\label{eq4li}
\end{eqnarray}
Plugging (\ref{eq3li}) and (\ref{eq4li}) into (\ref{eq2li}), and noticing that
$\eta=0$ on $\tilde\Gamma_\ll\setminus\partial\mathcal O$, we deduce
\begin{eqnarray}
  &&\int_{\mathcal O}[ \nabla_h(\tilde u\eta)\cdot \nabla_h\phi+a\partial_z(\tilde u\eta)\partial_z\phi] dx\nonumber\\
  &=&\int_{\mathcal O}
  [\tilde f\eta-( \nabla_h\cdot(\tilde u \nabla_h\eta)+ \nabla_h \tilde u\cdot \nabla_h\eta)-(\partial_z(a\partial_z\eta\tilde u)\nonumber\\
  &&+a\partial_z\tilde u\partial_z\eta)]\phi dx+\int_{\partial\mathcal O\cap \tilde\Gamma_\ll}(\tilde\varphi\eta+ \tilde u\partial_n\eta)\phi d\Gamma_\ll,\label{eq5li}
\end{eqnarray}
for any $\phi\in H^1(\tilde\Omega)$.

Define the truncation $U=\tilde u\eta$, and denote
\beq
F=\tilde f\eta-( \nabla_h\cdot(\tilde u \nabla_h\eta)+ \nabla_h \tilde u\cdot \nabla_h\eta)-(\partial_z(a\partial_z\eta\tilde u)+a\partial_z\tilde u\partial_z\eta),
\eeq
and
\beq
\Phi=\left\{
\begin{array}{lr}
\tilde u\partial_n\eta+\tilde\varphi\eta,&\mbox{on }\partial\mathcal O\cap\tilde\Gamma_\ll, \\
0,&\mbox{on }\partial\mathcal O\setminus\tilde\Gamma_\ll.
\end{array}
\right.
\eeq
Then, noticing that any function from $H^1(\mathcal O)$ can be extended
to be a function in $H^1(\tilde\Omega)$, (\ref{eq5li}) implies that
$U\in H^1(\mathcal O)$ is a weak solution to
\begin{equation}
  \left\{
  \begin{array}{lr}
  - \Delta_hU-\partial_z(\tilde a\partial_zU)=F,&\mbox{in }\mathcal O, \\
  \partial_{\mathcal N} U=\Phi,&\mbox{on }\partial\mathcal O,
  \end{array}
  \right.
\end{equation}
where $\mathcal N$ is the unit outward normal vector at the boundary $\partial\mathcal O$.

Since $\mathcal O$ is a smooth domain, one can now apply the classic regularity and elliptic estimates to the truncation $U$. Then, by the assumptions, one can check that
$F\in L^2(\Omega)$, and $\Phi\in H^{\frac12}(\partial\mathcal O)$, by the
trace theorem. Therefore, it follows from the regularity and elliptic
estimates for the elliptic equations subject to the Neumann boundary conditions in the bounded
smooth domains, that $U\in H^{2}(\mathcal O)$
and the following estimate holds
\begin{eqnarray}\label{eq5-1li}
  \|U\|_{H^{2}(\mathcal O)}\leq C(\|F\|_{L^2(\mathcal O)}+\|\Phi\|_{H^{\frac12}(\partial\mathcal O)}+\|U\|_{H^{1}(\mathcal O)}),
\end{eqnarray}
where the constant $C$ depends only on $\mathcal O$, the modulus of continuity of $a$, $\lambda$ and $\|\partial_za\|_{L^\infty(\mathcal O)}$.

We can now derive the regularity and elliptic estimate for $u$
from those for $U$.
Recalling that $\eta\equiv1$
on $\Omega$, one obtains $u\in H^{2}(\Omega)$, and it is obvious that
\beq
\|u\|_{H^{2}(\Omega)}\leq \|U\|_{H^{2}(\mathcal O)}.
\eeq
Note that
\beq
\|F\|_{L^2(\mathcal O)}\leq C(\|f\|_{L^2(\Omega)}+\|u\|_{H^{1}(\Omega)}),
\eeq
for a positive constant $C$ depending only on $\Omega$, $\|a\|_{L^\infty(\Omega)}$ and $\|\partial_za\|_{L^\infty(\Omega)}$,
\beq
\|U\|_{H^{1}(\mathcal O)}\leq C\|u\|_{H^{1}(\Omega)},
\eeq
for a positive constant $C$ depending only on $\Omega$, and by the trace inequality
\begin{eqnarray}
\|\Phi\|_{H^{\frac12}(\mathcal O)}&\leq& C\|\Phi\|_{H^{1}(\mathcal O)}
\leq C(\|u\|_{H^{1}(\Omega)}+\|\varphi\|_{H^{1}(\Omega)})\nonumber\\
&\leq&C(\|u\|_{H^{1}(\Omega)}+\|\varphi\|_{H^{\frac12}(\Gamma_\ll)}),
\end{eqnarray}
for a positive constant $C$ depending only on $\Omega$. In the above inequality, one recalls that the
boundary function $\varphi$ has already been extended to the whole domain,
at the beginning of the proof.

Thanks to the above estimates, it follows from (\ref{eq5-1li}) that
\beq
\|u\|_{H^{2}(\Omega)}\leq C(\|f\|_{L^2(\Omega)}+\|\varphi\|_{H^{\frac12}(\Gamma_\ll)}
+\|u\|_{H^{1}(\Omega)}),
\eeq
for a positive constant $C$ depending only on $\Omega$, the modulus of continuity of $a$, $\lambda$, $\|a\|_{L^\infty(\Omega)}$ and $\|\partial_za\|_{L^\infty(\mathcal O)}$. This completes the proof of
Lemma \ref{lem.reg.elliptic.special}.
\end{proof}

\begin{prop}[$H^2$ regularity: the general case]\label{prop.reg.elliptic.general}
Assume, in addition to the assumptions on $a,f,\varphi$ in Lemma \ref{lem.reg.elliptic.special}, that
\begin{eqnarray}
      b\in L^\infty(\Omega),\quad\alpha,\beta\in W^{1,\infty}(\Omega),\quad \psi\in H^{\frac12}(\Gamma_{01}),
\end{eqnarray}
where $\alpha$ and $\beta$ have been
extended to the whole domain $\Omega$.

Then, the unique weak solution $u$ stated in Proposition \ref{prop.weak.elliptic} belongs to
$H^{2}(\Omega)$, and the following estimate holds
\beq
\|u\|_{H^{2}(\Omega)}\leq C(\|f\|_{L^2(\Omega)}+\|\varphi\|_{H^{\frac12}(\Gamma_\ll)} + \|\psi\|_{H^{\frac12}(\Gamma_{01})}+\|u\|_{H^{1}(\Omega)}),
\eeq
for a positive constant $C$ depending only on $\Omega$, the modulus continuity
of $a$, $\lambda$, $\|a\|_{L^\infty(\Omega)}$, $\|\partial_za\|_{L^\infty(\Omega)}$, $\|b\|_{L^\infty(\Omega)}$, $\|\alpha\|_{W^{1,\infty}(\Omega)}$ and $\|\beta\|_{W^{1,\infty}(\Omega)}$.
\end{prop}

\begin{proof}
We set
\beq
\varphi_0=\varphi-\alpha u, \quad\psi_0=\psi-\beta u,\quad f_0=f-bu.
\eeq
Then it is obvious that $u$ is a weak solution to
\begin{equation}\label{eq5-2li}
  \left\{
  \begin{array}{lr}
    - \Delta_hu-\partial_z(a\partial_zu)=f_0,&\mbox{in }\Omega,\\
    \partial_nu=\varphi_0,&\mbox{on }\Gamma_\ll,\\
    \partial_\nu u=\psi_0,&\mbox{on }\Gamma_{01}.
  \end{array}
  \right.
\end{equation}
Since $\psi_0\in H^{\frac12}(\Gamma_{01})$ by the trace theorem there
is a function $\Psi_0\in H^{2}(\Omega)$, such that $\partial_\nu\Psi_0=\psi_0$
on $\Gamma_{01}$, and $\|\Psi_0\|_{H^{2}(\Omega)}\leq C\|\psi_0\|_{H^{\frac12}(\Gamma_{01})}$ for a positive constant $C$
depending only on $\Omega$. Setting $u_1=u-\Psi_0$, one can easily check that
$u_1$ is a weak solution to the following elliptic problem
\begin{equation}\label{eq5-3li}
  \left\{
  \begin{array}{lr}
    - \Delta_hu_1-\partial_z(a\partial_zu_1)=f_1,&\mbox{in }\Omega,\\
    \partial_nu_1=\varphi_1,&\mbox{on }\Gamma_\ll,\\
    \partial_\nu u_1=0,&\mbox{on }\Gamma_{01},
  \end{array}
  \right.
\end{equation}
where $f_1$ and $\varphi_1$ are given by
\beq
f_1=f_0+ \Delta_h\Psi_0+\partial_z(a\partial_z\Psi_0), \quad \varphi_1=\varphi_0- \partial_n\Psi_0.
\eeq
Note that $f_1\in L^2(\Omega)$ and $\varphi_1\in H^{\frac12}(\Gamma_\ll)$,
by Lemma \ref{lem.reg.elliptic.special}, one has $u_1\in H^{2}(\Omega)$, and the following estimate holds
\begin{equation}\label{eq6li}
\|u_1\|_{H^{2}(\Omega)}\leq C(\|f_1\|_{L^2(\Omega)}+\|\varphi_1\|_{H^{\frac12}(\Gamma_\ll)} +\|u_1\|_{H^{1}(\Omega)}),
\end{equation}
for a positive constant $C$ depending only on $\Omega$, the modulus continuity
of $a$, $\lambda$, $\|a\|_{L^\infty(\Omega)}$ and $\|\partial_za\|_{L^\infty(\Omega)}$.
Thus, recalling that $\Psi_0\in H^{2}(\Omega)$, it is clear that
$u=u_1+\Psi_0\in H^{2}(\Omega)$. Recalling that $\|\Psi_0\|_{H^{2}(\Omega)}\leq C\|\psi_0\|_{H^{\frac12}(\Gamma_{01})}$, for a positive constant $C$ depending only on $\Omega$, one deduces
\begin{eqnarray}
  \|u\|_{H^{2}(\Omega)}&\leq&\|u_1\|_{H^{2}(\Omega)} +\|\Psi_0\|_{H^{2}(\Omega)}\leq \|u_1\|_{H^{2}(\Omega)}+C\|\psi_0\|_{H^{\frac12}(\Gamma_{01})},\\
  \|u_1\|_{H^{1}(\Omega)}&\leq&\|u\|_{H^{1}(\Omega)} +\|\Psi_0\|_{H^{1}(\Omega)}\leq \|u\|_{H^{1}(\Omega)}+C\|\psi_0\|_{H^{\frac12}(\Gamma_{01})},\\
  \|\varphi_1\|_{H^{\frac12}(\Gamma_\ll)}&\leq&\|\varphi_0\|_{H^{\frac12} (\Gamma_\ll)} +\|\partial_n\Psi_0\|_{H^{\frac12}(\Gamma_\ll)}\leq \|\varphi_0\|_{H^{\frac12}(\Gamma_\ll)}+C\|\Psi_0\|_{H^{2}(\Omega)}\nonumber\\
  &\leq& C(\|\varphi_0\|_{H^{\frac12}(\Gamma_\ll)}+\|\psi_0
  \|_{H^{\frac12}(\Gamma_{01})}),
\end{eqnarray}
for a positive constant $C$ depending only on $\Omega$, and
\begin{eqnarray}
\|f_1\|_{L^2(\Omega)} &\leq& \|f\|_{L^2(\Omega)}+C\|u\|_{H^{1}(\Omega)}+C\|\Psi_0\|_{H^{2}(\Omega)}\nonumber\\
&\leq& C(\|f\|_{L^2(\Omega)}+\|u\|_{H^{1}(\Omega)}+ \|\psi_0\|_{H^{\frac12}(\Gamma_{01})}),
\end{eqnarray}
for a positive constant $C$ depending only on $\Omega$, $\|a\|_{L^\infty(\Omega)}$, $\|\partial_za\|_{L^\infty(\Omega)}$ and $\|b\|_{L^\infty(\Omega)}$.
Thanks to the above estimates, one derives from (\ref{eq6li}) that
\begin{equation}\label{eq7li}
\|u\|_{H^{2}(\Omega)}\leq C(\|f\|_{L^2(\Omega)}+\|\varphi_0\|_{H^{\frac12}(\Gamma_\ll)} +\|\psi_0\|_{H^{\frac12}(\Gamma_{01})}+\|u\|_{H^{1}(\Omega)}),
\end{equation}
for a positive constant $C$ depending only on $\Omega$, $\|a\|_{L^\infty(\Omega)}$, $\|\partial_za\|_{L^\infty(\Omega)}$ and $\|b\|_{L^\infty(\Omega)}$. One still need to estimate $\|\varphi_0\|_{H^{\frac12}(\Gamma_\ll)} +\|\psi_0\|_{H^{\frac12}(\Gamma_{01})}$, which is done as follows. By the
trace inequality, one has
\begin{eqnarray}
&&\|\varphi_0\|_{H^{\frac12}(\Gamma_\ll)} +\|\psi_0\|_{H^{\frac12}(\Gamma_{01})}\nonumber\\
&\leq&\|\varphi \|_{H^{\frac12}(\Gamma_\ll)}  +\|\psi \|_{H^{\frac12}(\Gamma_{01})}+\|\alpha u\|_{H^{\frac12}(\Gamma_\ll)}
+\|\beta u\|_{H^{\frac12}(\Gamma_{01})}\nonumber\\
&\leq& \|\varphi \|_{H^{\frac12}(\Gamma_\ll)}  +\|\psi \|_{H^{\frac12}(\Gamma_{01})}+C(\|\alpha u\|_{H^{1}(\Omega)}
+\|\beta u\|_{H^{1}(\Omega)})\nonumber\\
&\leq& \|\varphi \|_{H^{\frac12}(\Gamma_\ll)}  +\|\psi \|_{H^{\frac12}(\Gamma_{01})}+C \|  u\|_{H^{1}(\Omega)},
\end{eqnarray}
for a positive constant $C$ depending only on $\Omega$,  $\|\alpha\|_{W^{1,\infty}(\Omega)}$ and $\|\beta\|_{W^{1,\infty}(\Omega)}$.
Plugging the above estimate into (\ref{eq7li}) yields the conclusion.
\end{proof}

\subsection{Linear parabolic equations}
Given a positive time $\mathcal T$, set $Q_{\mathcal T}=\Omega\times(0,\mathcal T)$. Consider the parabolic problem
\begin{equation}
  \label{paraboic1}
  \left\{
  \begin{array}{lr}
  \partial_tu+\mathcal Lu=f,&\mbox{in }Q_{\mathcal T}, \\
  \mathcal Bu=\Phi,&\mbox{on }\partial\Omega\times(0,\mathcal T),\\
  u(\cdot,0)=u_0,&\mbox{on }\Omega,
  \end{array}
  \right.
\end{equation}
where the elliptic operator $\mathcal L$ is given by
\beq
\mathcal Lu=- \Delta_h u-\partial_z( a\partial_z u),
\eeq
the boundary operator $\mathcal B$ is given by
\beq
\mathcal Bu=
\left\{
\begin{array}{lr}
\partial_nu+\alpha u,&\mbox{on }\Gamma_\ll, \\
\partial_\nu u+\beta u,&\mbox{on }\Gamma_{01},
\end{array}
\right.
\eeq
and the boundary function $\Phi$ is given by
\beq
\Phi=\left\{
\begin{array}{lr}
\varphi,&\mbox{ on }\Gamma_\ll,\\
\psi,&\mbox{ on }\Gamma_{01},
\end{array}
\right.
\eeq
for some functions $\varphi$ and $\psi$.

Throughout this subsection, we assume that the coefficients $a, \alpha$ and $\beta$ in the elliptic
operator $\mathcal L$ and the boundary operator $\mathcal B$ are independent of the time variable $t$, while the nonhomogeneous functions $\varphi$ and
$\psi$ are allowed to depend on $t$. Assume that the coefficients $a, \alpha$ and $\beta$ satisfy
\begin{eqnarray}\label{asum1}
a\in C(\overline\Omega),\quad\partial_za\in L^\infty(\Omega), \quad \lambda\leq a(x)\leq\Lambda,\quad 0\leq\alpha,\beta\in W^{1,\infty}(\Omega),
\end{eqnarray}
for some positive constants $\lambda$ and $\Lambda$,  where the functions $\alpha, \beta$, defined on the boundary $\partial\Omega$, have been extended to the whole domain $\Omega$. We assume moreover  that the boundary functions $\varphi, \psi$ and $\Phi$ satisfy
\begin{equation}\label{asum2}
\varphi\in L^2(0,T; H^{\frac12}(\Gamma_\ll)),\quad\psi\in L^2(0,T;H^{\frac12}(\Gamma_{01})),\quad\partial_t\Phi\in L^2(\partial\Omega\times(0,T)).
\end{equation}

We are going to prove the existence and uniqueness of weak and strong solutions (see the definitions below) to (\ref{paraboic1}).

\begin{df}
Given a positive time $\mathcal T\in(0,\infty)$ and a function $u_0\in L^2(\Omega)$.
Assume that (\ref{asum1})
and (\ref{asum2}) hold, and $f\in L^2(Q_{\mathcal T})$. A function $u$ is called a weak solution to (\ref{paraboic1}), if $u\in L^\infty(0,T; L^2(\Omega))\cap L^2(0,T; H^1(\Omega))$, and the following integral equality holds
\beq
\int_0^T[-(u,\partial_t\xi)+\langle u,\xi\rangle_a]dt=\int_0^Tb(f,\varphi,\psi,\xi) dt+(u_0,\xi(\cdot, 0)),
\eeq
for any $\xi\in H^1(Q_{\mathcal T})$, with $\xi(\cdot,\mathcal T)\equiv0$, where $(\cdot,\cdot)$
is the $L^2(\Omega)$ inner product,
\beq
\langle u,\xi \rangle_a=\int_\Omega( \nabla_hu\cdot\nabla\xi+a\partial_zu\partial_z\xi)dx
+\int_{\Gamma_\ll}\alpha u\xi d\Gamma_\ll+\int_{\Gamma_{01}}a\beta u\xi dx',
\eeq
and
\beq
b(f,\varphi,\psi,\xi)=\int_\Omega f\xi dx+\int_{\Gamma_\ll}\varphi\xi d\Gamma_\ll+\int_{\Gamma_{01}}\psi\xi dx'.
\eeq
\end{df}

\begin{df}
Given a positive time $\mathcal T\in(0,\infty)$ and a function $u_0\in H^1(\Omega)$. Assume that (\ref{asum1})
and (\ref{asum2}) hold, and let $f\in L^2(Q_{\mathcal T})$. A function $u$ is called a strong solution to (\ref{paraboic1}), if
\beq
u\in C([0,\mathcal T]; H^1(\Omega))\cap L^2(0,\mathcal T; H^2(\Omega)),\quad\partial_tu\in L^2(0,\mathcal T; L^2(\Omega)),
\eeq
the equation in (\ref{paraboic1}) is satisfied, a.e.\,in $Q_{\mathcal T}$, the boundary
condition in (\ref{paraboic1}) is satisfied in the sense of trace, and the
initial condition in (\ref{paraboic1}) is fulfilled.
\end{df}
Let us first transform the nonhomogeneous
boundary value problem to the homogeneous one.
For each $t\in[0,\mathcal T]$, we define $\mathcal U_\Phi(\cdot,t)$ as the unique solution to
\begin{equation}
  \left\{
  \begin{array}{lr}
  \mathcal L\mathcal U_\Phi=0, &\mbox{in }\Omega,\\
  \mathcal B\mathcal U_\Phi=\Phi, &\mbox{on }\partial\Omega.
  \end{array}
  \right.
\end{equation}
Noticing that $\Phi\in C([0,\mathcal T]; L^2(\partial\Omega))$, by applying Proposition \ref{prop.weak.elliptic} and Proposition \ref{prop.reg.elliptic.general}, one can see that
$\mathcal U_\Phi\in C([0,\mathcal T]; H^1(\Omega))\cap L^2(0,\mathcal T; H^2(\Omega))$, and
\begin{align*}
\|\mathcal U_\Phi\|_{L^\infty(0,\mathcal T;H^1(\Omega))}+\|\mathcal U_\Phi\|_{L^2(0,\mathcal T; H^2(\Omega))}\leq&C\big(\|\varphi\|_{L^2(0,\mathcal T;H^{\frac12}(\Gamma_\ll))} +\|\psi\|_{L^2(0,\mathcal T;H^{\frac12}(\Gamma_{01}))}\big),
\end{align*}
for a positive constant $C$ depending only on $\Omega$, $\lambda$, the modulus of continuity of $a$, $\|a\|_{L^\infty(\Omega)}$, $\|\partial_za\|_{L^\infty(\Omega)}$, $\|\alpha\|_{W^{1,\infty} (\Omega)}$ and $\|\beta\|_{W^{1,\infty}(\Omega)}$. Note that $\partial_t\mathcal U_\Phi$ satisfies
\begin{equation}
  \left\{
  \begin{array}{lr}
  \mathcal L\mathcal \partial_t\mathcal U_\Phi=0, &\mbox{in }\Omega,\\
  \mathcal B\mathcal \partial_t\mathcal U_\Phi=\partial_t\Phi, &\mbox{on }\partial\Omega,
  \end{array}
  \right.
\end{equation}
by Proposition \ref{prop.weak.elliptic}, it follows
\beq
\|\partial_t\mathcal U_\Phi\|_{L^2(0,T; H^1(\Omega))}\leq C\|\partial_t\Phi\|_{L^2(0,T; L^2(\partial\Omega))}.
\eeq
Therefore, we have
\begin{eqnarray}
  &&\|\mathcal U_\Phi\|_{L^\infty(0,\mathcal T;H^1(\Omega))}+\|\mathcal U_\Phi\|_{L^2(0,\mathcal T; H^2(\Omega))}+\|\partial_t\mathcal U_\Phi\|_{L^2(0,\mathcal T; H^1(\Omega))}\nonumber\\
  &\leq&C\big(\|\varphi\|_{L^2(0,\mathcal T;H^{\frac12}(\Gamma_\ll))} +\|\psi\|_{L^2(0,\mathcal T;H^{\frac12}(\Gamma_{01}))}+\|\partial_t\Phi\|_{L^2(0,\mathcal T; L^2(\partial\Omega))}\big),\label{5.28.1}
\end{eqnarray}
for a positive constant $C$ depending only on $\Omega$, $\lambda$, the modulus of continuity of $a$, $\|a\|_{L^\infty(\Omega)}$, $\|\partial_za\|_{L^\infty(\Omega)}$, $\|\alpha\|_{W^{1,\infty} (\Omega)}$ and $\|\beta\|_{W^{1,\infty}(\Omega)}$.

Suppose that $u$ is a weak (strong) solution to (\ref{paraboic1}), by setting $v=u-\mathcal U_\Phi$, one can easily verify that $v$ is a weak (strong) solution to the following problem
\begin{equation}
  \label{paraboic2}
  \left\{
  \begin{array}{lr}
  \partial_tv+\mathcal Lv=g,&\mbox{in }Q_{\mathcal T}, \\
  \mathcal Bv=0,&\mbox{on }\partial\Omega\times(0,\mathcal T),\\
  v(\cdot,0)=v_0,&\mbox{on }\Omega,
  \end{array}
  \right.
\end{equation}
where
\beq
g=f-\partial_t\mathcal U_\Phi, \quad v_0=u_0-\mathcal U_\Phi(\cdot,0).
\eeq
Conversely, if $v$ is a weak (strong) solution to (\ref{paraboic2}), the $u=v+\mathcal U_\Phi$ is a weak (strong) solution to (\ref{paraboic1}).
Therefore, to prove the existence and uniqueness of weak (strong) solutions
to (\ref{paraboic1}), it suffices to prove the corresponding results for
(\ref{paraboic2}).

We are going to construct a sequence of approximated solutions $v_N$ to the
parabolic problem (\ref{paraboic2}) by discretizing in time.
Let $g\in L^2(Q_{\mathcal T})$. We fix an arbitrary
positive integer $N$, and let $h=\frac{\mathcal T}{N}$. We
set
\begin{equation}\label{fkdf}
g^{k+1}=\frac1h\int_{kh}^{(k+1)h}g(t) dt, \quad k=0,1,\cdots,N-1.
\end{equation}
Due to the assumption, it is obvious that $g^{k+1}\in L^2(\Omega)$ for
$k=0,1,\cdots,N-1.$ It follows from the Minkowski and H\"older inequalities
that
\begin{eqnarray}
  h\sum_{k=0}^{N-1}\|g^{k+1}\|_{L^2(\Omega)}&=&\sum_{k=0}^{N-1}\left\| \int_{kh}^{(k+1)h}g(t)dt\right\|_{L^2(\Omega)}\leq\sum_{k=0}^{N-1}  \int_{kh}^{(k+1)h}\|g(\cdot,t)\|_{L^2(\Omega)}dt\nonumber\\
  &=&\int_0^\mathcal T\|g(\cdot,t)\|_{L^2(\Omega)}dt=\|g\|_{L^1(0,\mathcal T; L^2(\Omega))} \leq\sqrt \mathcal T\|g\|_{L^2 (Q_{\mathcal T})}, \label{f1-1}
\end{eqnarray}
and
\begin{eqnarray}
  h\sum_{k=0}^{N-1}\|g^{k+1}\|_{L^2(\Omega)}^2&=&\frac1h \sum_{k=0}^{N-1}\left\| \int_{kh}^{(k+1)h}g(t)dt\right\|_{L^2(\Omega)}^2\nonumber\\
  &\leq&\frac1h\sum_{k=0}^{N-1}  \left(\int_{kh}^{(k+1)h}\|g(\cdot,t)\|_{L^2(\Omega)}dt\right)^2\nonumber\\
  &\leq&\int_0^\mathcal T\|g(\cdot,t)\|_{L^2(\Omega)}^2dt=\|g\|_{L^2 (Q_{\mathcal T})}^2.\label{f1-2}
\end{eqnarray}

We set $v^0=v_0$ and define
$v^{k+1}$ for $k=0, 1,\cdots,N-1$ successively as the unique solution to the
elliptic problem
\begin{equation}
  \label{paeq1}
  \left\{
  \begin{array}{lr}
  \frac{v^{k+1}-v^k}{h}+\mathcal Lv^{k+1}=g^{k+1},&\mbox{in }\Omega,\\
  \mathcal Bv^{k+1}=0,&\mbox{on }\partial\Omega.
  \end{array}
  \right.
\end{equation}
According to Proposition \ref{prop.weak.elliptic} and Proposition \ref{prop.reg.elliptic.general} there is a unique solution $v^{k+1}\in H^2(\Omega)$ to
(\ref{paeq1}) for $k=0,1,\cdots,N-1$. Define the approximate solution $v_N$ as
\begin{equation}\label{uN}
  v_N(t)=\sum_{k=0}^{N-1}\chi_{[kh,(k+1)h)}(t)v^{k+1},\quad t\in[0,T).
\end{equation}

Some a priori estimates on $v_N$ are stated in the following lemma.

\begin{lemma}\label{parapri}
Assume that (\ref{asum1}) holds. Set $v^0=v_0\in L^2(\Omega)$, and let $v^{k+1}$, $k=0,1,\cdots,N-1$, be the functions defined
by (\ref{paeq1}), and $v_N$ by (\ref{uN}). Then, the following estimates hold:

(i) There is a positive constant $C$ depending only on $\lambda$ and $T$, such that
\beq
\|v_N\|_{L^\infty(0,\mathcal T; L^2(\Omega))}+\|v_N\|_{L^2(0,\mathcal T;H^1(\Omega))} \leq C(\|g\|_{L^2(Q_{\mathcal T})}+\|v_0\|_{L^2(\Omega)});
\eeq

(ii) If we assume in addition that $v_0\in H^1(\Omega)$, then we have
\beq
\|v_N\|_{L^\infty(0,\mathcal T; H^1(\Omega))}+\|v_N\|_{L^2(0,\mathcal T;H^2(\Omega))}\leq C(\|g\|_{L^2(Q_{\mathcal T})}+\|v_0\|_{H^1(\Omega)}),
\eeq
for a positive constant $C$ depending only on $\Omega$, $\lambda$, $T$,
the modulus of continuity
of $a$, $\|a\|_{L^\infty(\Omega)}$, $\|\partial_za\|_{L^\infty(\Omega)}$, $\|\alpha\|_{W^{1,\infty}(\Omega)}$ and $\|\beta\|_{W^{1,\infty}(\Omega)}$.
\end{lemma}

\begin{proof}
(i) Testing (\ref{paeq1}) by $v^{k+1}$, and summing the
results with respect to $k$ from $0$ to $M$ yields
\begin{equation}
  h\sum_{k=0}^{M}\langle v^{k+1}, v^{k+1}\rangle_a+\sum_{k=0}^{M}(v^{k+1}-v^k,v^{k+1}) =h\sum_{k=0}^{M}(g^{k+1}, v^{k+1}), \label{paineq1}
\end{equation}
for $M=0,1,\cdots,N-1$.
Straightforward calculations yield
\begin{eqnarray}
  \sum_{k=0}^{M}(v^{k+1}-v^k, v^{k+1})&=&\frac12\left(\|v^{M+1} \|_{L^2(\Omega)}^2 -\|v_0\|_{L^2(\Omega)}^2+\sum_{k=0}^{M} \|v^{k+1}-v^k\|_{L^2(\Omega)}^2 \right)\nonumber\\
  &\geq& \frac12\left(\|v^{M+1} \|_{L^2(\Omega)}^2 -\|v_0\|_{L^2(\Omega)}^2\right). \label{paineq2}
\end{eqnarray}
It follows from the Cauchy inequality that
\begin{eqnarray}
  \sum_{k=0}^M(g^{k+1},v^{k+1})
  &\leq&\sup_{0\leq k\leq M}\|v^{k+1}\|_{L^2(\Omega)}  \sum_{k=0}^M\|g^{k+1}\|_{L^2(\Omega)}\nonumber\\
  &\leq&\sup_{0\leq k\leq N-1}\|v^{k+1}\|_{L^2(\Omega)}  \sum_{k=0}^{N-1}\|g^{k+1}\|_{L^2(\Omega)}.\label{paineq3}
\end{eqnarray}
Thanks to (\ref{paineq2}) and (\ref{paineq3}) it
follows from (\ref{paineq1}) using (\ref{f1-1}) that
\begin{eqnarray}
  &&h\sum_{k=0}^{M}\langle v^{k+1}, a^{k+1}\rangle_a+\frac12\|v^{M+1}\|_{L^2(\Omega)}^2 \nonumber\\
  &\leq&h\sup_{0\leq k\leq N-1}\|v^{k+1}\|_{L^2(\Omega)}  \sum_{k=0}^{N-1}\|g^{k+1}\|_{L^2(\Omega)}+\frac12\|v_0\|_{L^2(\Omega)}^2\nonumber\\
  &\leq&\sup_{0\leq k\leq N-1}\|v^{k+1}\|_{L^2(\Omega)}
  \sqrt \mathcal T\|g\|_{L^2 (Q_{\mathcal T})}+\frac12\|v_0\|_{L^2(\Omega)}^2, \label{paineq4}
\end{eqnarray}
for $M=0,1,\cdots,N-1$. Setting
\beq
A=\sup_{0\leq M\leq N-1}\left( h\sum_{k=0}^{M}\langle v^{k+1}, v^{k+1}\rangle_a+\frac12\|v^{M+1}\|_{L^2(\Omega)}^2 \right),
\eeq
it follows from the Young inequality and (\ref{paineq4}) that
\begin{eqnarray}
  A&\leq& \sqrt{2A}\sqrt\mathcal T\|g\|_{L^2 (Q_{\mathcal T})}+\frac12\|v_0\|_{L^2(\Omega)}^2\nonumber\\
  &\leq&\frac12A+  \mathcal T\|g\|_{L^2 (Q_{\mathcal T})}^2 + \frac12\|v_0\|_{L^2(\Omega)}^2,
\end{eqnarray}
and thus
\beq
A\leq 2\mathcal T\|g\|_{L^2 (Q_{\mathcal T})}^2  + \|v_0\|_{L^2(\Omega)}^2.
\eeq
Thanks to this, and recalling that $\langle v, v\rangle_a\geq\lambda\|\nabla v\|_{L^2(\Omega)}^2$, one obtains
\begin{equation}
   h\lambda\sum_{k=0}^{N-1}\|\nabla v^{k+1}\|_{L^2(\Omega)}^2
  +\sup_{0\leq k\leq N-1} \|v^{k+1}\|_{L^2(\Omega)}^2
  \leq  8\mathcal T\|g\|_{L^2 (Q_{\mathcal T})}^2  +4 \|v_0\|_{L^2(\Omega)}^2,\label{estpa1}
\end{equation}
from which, recalling the definition of $v_N$, (i) follows.

(ii) Testing (\ref{paeq1}) by $v^{k+1}-v^k$, and summing the resultants with
respect to $k$ from $0$ to $M$, for $M=0,1,\cdots,N-1$, yields
\begin{equation}
  \sum_{k=0}^M\left(\langle v^{k+1}, v^{k+1}-v^k\rangle_a+\frac1h\|v^{k+1}-v^k\|_{L^2(\Omega)}^2\right)= \sum_{k=0}^M(g^{k+1},v^{k+1}-v^k),
\end{equation}
from which, noticing that
\begin{equation}
  \sum_{k=0}^M\langle v^{k+1},v^{k+1}-v^k\rangle_a=\frac12 \left(\langle v^{M+1},v^{M+1}\rangle_a- \langle v_0, v_0\rangle_a
  +\sum_{k=0}^M\langle v^{k+1}-v^k,v^{k+1}-v^k\rangle_a\right),
\end{equation}
one obtains
\begin{eqnarray}
   \langle v^{M+1},v^{M+1}\rangle_a
  +\frac2h\sum_{k=0}^M \| {v^{k+1}-v^k} \|_{L^2(\Omega)}^2
  \leq2\sum_{k=0}^M(g^{k+1},v^{k+1}-v^k)+\langle v_0, v_0\rangle_a,
\end{eqnarray}
for $M=0,1,\cdots,N-1$.
Applying the Cauchy inequality to the righthand side of the above equality, and using (\ref{f1-2}), one obtains
\begin{eqnarray}
  2\sum_{k=0}^M(g^{k+1},v^{k+1}-v^k)&\leq& 2\sum_{k=0}^M\|g^{k+1}\|_{L^2(\Omega)}\|v^{k+1}-v^k\|_{L^2(\Omega)}\nonumber\\
  &\leq&\frac1h\sum_{k=0}^M\|v^{k+1}-v^k\|_{L^2(\Omega)}^2 +h\sum_{k=0}^M\|g^{k+1}\|_{L^2(\Omega)}^2\nonumber\\
  &\leq&\frac1h\sum_{k=0}^M\|v^{k+1}-v^k\|_{L^2(\Omega)}^2 +\|g\|_{L^2(Q_{\mathcal T})}^2,
\end{eqnarray}
which, plugged into the previous inequality, yields
\begin{eqnarray}
   \langle v^{M+1},v^{M+1}\rangle_a
  +h\sum_{k=0}^M\left\|\frac{v^{k+1}-v^k}{h} \right\|_{L^2(\Omega)}^2
  \leq \|g\|_{L^2(Q_{\mathcal T})}^2+\langle v_0, v_0\rangle_a,
\end{eqnarray}
for $M=0,1,\cdots,N-1$, and thus
\begin{equation}\label{estpa2}
  \sup_{1\leq k\leq N}\langle v^k, v^k\rangle_a+h\sum_{k=0}^{N-1}\left\|\frac{v^{k+1}-v^k}{h} \right\|_{L^2(\Omega)}^2  \leq \|g\|_{L^2(Q_{\mathcal T})}^2+\langle v_0, v_0\rangle_a.
\end{equation}

Combining (\ref{f1-2}), (\ref{estpa1}) and (\ref{estpa2}),
and applying the $H^2$ estimate to the elliptic operator $\mathcal L$, there is a positive
constant $C$ depending only on $\Omega$, $\lambda$, the modulus of continuity
of $a$, $\|a\|_{L^\infty(\Omega)}$, $\|\partial_za\|_{L^\infty(\Omega)}$, $\|\alpha\|_{W^{1,\infty}(\Omega)}$ and $\|\beta\|_{W^{1,\infty}(\Omega)}$, such that
\begin{eqnarray}
  h\sum_{k=1}^{N-1}\|v^{k+1}\|_{H^2(\Omega)}^2&\leq&C h\sum_{k=1}^{N-1}(\|\mathcal Lv^{k+1}\|_{L^2(\Omega)}^2+\|v^{k+1}\|_{H^1(\Omega)}^2) \nonumber\\
  &\leq&Ch\sum_{k=1}^{N-1} \left(\left\|g^{k+1}-\frac{v^{k+1}-v^k}{h}\right\|_{L^2(\Omega)}^2 +\|v^{k+1}\|_{H^1(\Omega)}^2\right)\nonumber\\
  &\leq&C(\mathcal T^2+1)(\|g\|_{L^2(Q_{\mathcal T})}^2+\|v_0\|_{H^1(\Omega)}^2),
\end{eqnarray}
from which, recalling the definition of $v_N$ and (\ref{estpa2}), (ii) follows.
\end{proof}

The existence and uniqueness of weak and strong solutions to (\ref{paraboic2}) is now stated in the following proposition.

\begin{prop}\label{prop.weak.str.par.hom}
Given a positive time $\mathcal T\in(0,\infty)$ and the initial data $v_0\in L^2(\Omega)$, we
assume that (\ref{asum1}) holds, and that $g\in L^2(Q_{\mathcal T})$.

Then, there is a unique weak solution to (\ref{paraboic2}), satisfying
\begin{align*}
\|v\|_{L^\infty(0,\mathcal T; L^2(\Omega))}+\|v\|_{L^2(0,\mathcal T; H^1(\Omega))}
\leq& C(\|g\|_{L^2(Q_{\mathcal T})} +\|v_0\|_{L^2(\Omega)}),
\end{align*}
for a positive constant $C$ depending only on $\lambda$ and $\mathcal T$.

Moreover, if we assume in addition that $v_0\in H^1(\Omega)$, then the unique weak solution is a strong one, and satisfies
\begin{equation}
\|v\|_{L^\infty(0,\mathcal T; H^1(\Omega))}+\|v\|_{L^2(0,\mathcal T; H^2(\Omega))}+\|\partial_tv\|_{L^2(Q_\mathcal T)}\leq C(\|g\|_{L^2(Q_{\mathcal T})}+\|v_0\|_{H^1(\Omega)}),
\end{equation}
for a positive constant $C$ depending only on $\Omega$, $\lambda$,
$\mathcal T$, the modulus of continuity of $a$, $\|a\|_{L^\infty(\Omega)}$,
$\|\partial_za\|_{L^\infty(\Omega)}$, $\|\alpha\|_{W^{1,\infty}(\Omega)}$
and $\|\beta\|_{W^{1,\infty}(\Omega)}$.
\end{prop}

\begin{proof}
Let $g^{k+1}$, $k=0,1,\cdots,N-1$, be the functions defined by (\ref{fkdf}), and set
\beq
g_N(t)=\sum_{k=0}^{N-1}\chi_{[kh,(k+1)h)}(t)g^{k+1}.
\eeq
Recalling that $g\in L^2(0,\mathcal T; L^2(\Omega))$, one can verify that $g_N\rightarrow g$, in $L^2(0,\mathcal T; L^2(\Omega))$, as $N\rightarrow\infty$. Let $v^{k+1}$, $k=0, 1, \cdots, N-1$, be the functions defined by (\ref{paeq1}). Let $v_N$ be the function given by (\ref{uN}).

By Lemma \ref{parapri}, the following estimate holds
\beq
\|v_N\|_{L^\infty(0,\mathcal T; L^2(\Omega))}+\|v_N\|_{L^2(0,\mathcal T; H^1(\Omega))}
\leq C(\|g\|_{L^2(Q_{\mathcal T})} +\|v_0\|_{L^2(\Omega)} ),
\eeq
for a positive constant $C$ depending only on $\lambda$ and $\mathcal T$; and if we assume
in addition that $v_0\in H^1(\Omega)$, then the following additional
estimate holds
\beq
\|v_N\|_{L^\infty(0,\mathcal T; H^1(\Omega))}+\|v_N\|_{L^2(0,\mathcal T; H^2(\Omega))}\leq C(\|g\|_{L^2(Q_{\mathcal T})}+\|v_0\|_{H^1(\Omega)}),
\eeq
for a positive constant $C$ depending only on $\Omega$, $\lambda$,
$\mathcal T$, the modulus of continuity of $a$, $\|a\|_{L^\infty(\Omega)}$,
$\|\partial_za\|_{L^\infty(\Omega)}$, $\|\alpha\|_{W^{1,\infty}(\Omega)}$
and $\|\beta\|_{W^{1,\infty}(\Omega)}$. Thanks to the above estimates,
there is a subsequence, still denoted by $\{v_N\}_{N=1}^\infty$,
and a function $v$, with $v\in L^\infty(0,\mathcal T; L^2(\Omega))\cap L^2(0,\mathcal T; H^1(\Omega))$,
such that
\begin{equation}\label{5.28.2}
\|v\|_{L^\infty(0,\mathcal T; L^2(\Omega))}+\|v\|_{L^2(0,\mathcal T; H^1(\Omega))}
\leq C(\|g\|_{L^2(Q_{\mathcal T})} +\|v_0\|_{L^2(\Omega)} ),
\end{equation}
for a positive constant $C$ depending only on $\lambda$ and $\mathcal T$, and
\beq
v_N\overset{*}{\rightharpoonup}v\mbox{ in }L^\infty(0,\mathcal T; L^2(\Omega)), \mbox{ and }v_N\rightharpoonup v\mbox{ in }L^2(0,\mathcal T; H^1(\Omega)).
\eeq
Moreover, if we assume in addition that $v_0\in H^1(\Omega)$, then the above
$v$ has the additional regularity that $v\in L^2(0,\mathcal T; H^2(\Omega))
\cap L^\infty(0,\mathcal T; H^1(\Omega))$, and
\begin{equation}\label{5.28.3}
\|v\|_{L^\infty(0,\mathcal T; H^1(\Omega))}+\|v\|_{L^2(0,\mathcal T; H^2(\Omega))}\leq C(\|g\|_{L^2(Q_{\mathcal T})}+\|v_0\|_{H^1(\Omega)}),
\end{equation}
for a positive constant $C$ depending only on $\Omega$, $\lambda$,
$\mathcal T$, the modulus of continuity of $a$, $\|a\|_{L^\infty(\Omega)}$,
$\|\partial_za\|_{L^\infty(\Omega)}$, $\|\alpha\|_{W^{1,\infty}(\Omega)}$
and $\|\beta\|_{W^{1,\infty}(\Omega)}$.

We take an arbitrary function $\eta\in C^1(\overline{Q_{\mathcal T}})$ with $\eta(\cdot,t)\equiv0$ when $t$ is close enough to $\mathcal T$, and set
$\eta^k=\eta(\cdot, hk)$, $k=0,1,\cdots,N-1$.
Then, for large enough $N$, one has $\eta^{N-1}=\eta(\cdot, (1-\frac1N)\mathcal T)\equiv0$. Define two functions $\eta_N$ and $\bar{\partial_{ t}}\eta_N$ as
\beq
\eta_N=\sum_{k=0}^{N-1}\chi_{[kh,(k+1)h)}(t)\eta^k,\quad\bar{\partial_{ t}}\eta_N=\sum_{k=0}^{N-1}\chi_{[kh,(k+1)h)}(t)\frac{\eta^{k+1}-\eta^{k}}{h}.
\eeq
Then, using $\eta\in C^1(\overline{Q_{\mathcal T}})$, one can show that
\beq
\eta_N\rightarrow\eta, \mbox{ in }L^2(0,\mathcal T; H^1(\Omega)),\quad\mbox{and}\quad\bar{\partial_{ t}}\eta_N\rightarrow\partial_t\eta,\mbox{ in }L^2(Q_{\mathcal T}).
\eeq

Taking $\eta^k$ as the testing function, and summing the resultant with
respect to $k$ from $0$ to $N-1$ yields
\begin{eqnarray}\label{paeq3}
  h\sum_{k=0}^{N-1}\langle v^{k+1},\eta^k\rangle_a+\sum_{k=0}^{N-1}(v^{k+1}-v^k,\eta^k) =h\sum_{k=0}^{N-1}(g^{k+1}, \eta^k).
\end{eqnarray}
Direct calculations yield
\begin{eqnarray}
  \sum_{k=0}^{N-1}(v^{k+1}-v^k,\eta^k)&=&(v^N, \eta^{N-1})-(v_0,\eta^0)-\sum_{k=1}^{N-1}(v^k,\eta^k-\eta^{k-1})\nonumber\\
  &=&-(v_0,\eta^0)-h\sum_{k=1}^{N-1}(v^k,\frac{\eta^k-\eta^{k-1}}{h})\nonumber\\
  &=&-(v_0,\eta^0)-h\sum_{k=0}^{N-2}(v^{k+1}, \frac{\eta^{k+1}-\eta^{k}}{h}),
\end{eqnarray}
from which, noticing that $\eta^{N}=\eta^{N-1}\equiv0$, one obtains
\beq
\sum_{k=0}^{N-1}(v^{k+1}-v^k,\eta^k)=-(v_0,\eta^0)-h\sum_{k=0}^{N-1}(v^{k+1}, \frac{\eta^{k+1}-\eta^{k}}{h}).
\eeq
Thanks to the above equality, it follows from (\ref{paeq3}) that
\begin{eqnarray}
  -h\sum_{k=0}^{N-1}(v^{k+1},\frac{\eta^{k+1}-\eta^{k}}{h}) +h\sum_{k=0}^{N-1}\langle v^{k+1},\eta^k \rangle_a =h\sum_{k=0}^{N-1}(g^{k+1}, \eta^k)+(v_0,\eta^0).
\end{eqnarray}

Recalling the definitions of $v_N, g_N, \eta_N$ and $\partial_{\bar t}\eta_N$, one can verify that the above equality is equivalent to
\beq
\int_0^T(-(v_N,\bar{\partial_{ t}}\eta_N)+\langle v_N, \eta_N\rangle_a)dt=\int_0^T(g_N, \eta_N) dt+(v_0,\eta(\cdot,0)),
\eeq
from which by letting $N\rightarrow \infty$ due to  $\eta_N\rightarrow\eta$ in $L^2(0,\mathcal T; H^1(\Omega))$ and $\bar{\partial_{ t}}\eta_N\rightarrow\partial_t\eta$ in $L^2(Q_{\mathcal T})$ one obtains
\beq
\int_0^T(-(v,\bar{\partial_{ t}}\eta)+\langle v, \eta\rangle_a )dt=\int_0^T(g, \eta) dt+(v_0,\eta(\cdot,0)).
\eeq
Therefore, $v$ is a weak solution to (\ref{paraboic2}) fulfilling the estimate (\ref{5.28.2}). Moreover, if in addition $v_0\in H^1(\Omega)$, one has the additional regularities that $v\in L^2(0,\mathcal T; H^2(\Omega))\cap L^\infty(0,\mathcal T; H^1(\Omega))$, and the estimate (\ref{5.28.3}) holds. This proves the existence part of the proposition, while the uniqueness can be proven in the
standard way, see e.g. Ladyzhenskaya et al.\,\cite{LSU}.

We now prove that the weak solutions just established are strong ones and satisfy the corresponding estimate,
if $v_0\in H^1(\Omega)$. Recall that, in this case,
one has $v\in L^2(0,\mathcal T; H^2(\Omega))\cap L^\infty(0,\mathcal T; H^1(\Omega))$, and (\ref{5.28.3}) holds.
Thanks to this fact, and noticing that the
equation in (\ref{paraboic2}) is satisfied in the sense of distribution,
one obtains $\partial_tv\in L^2(Q_{\mathcal T})$, which, along with
$v\in L^2(0,\mathcal T; H^2(\Omega))$ and $g\in L^2(Q_{\mathcal T})$, in turn implies that
the equation in (\ref{paraboic2}) is
satisfied, a.e.\,in $Q_{\mathcal T}$. Thus, recalling the estimate (\ref{5.28.3}), one obtains the desired $L^2(Q_\mathcal T)$ estimate for $\partial_tv$,
stated in the proposition. The regularities $v\in L^2(0,\mathcal T; H^2(\Omega))$ and $\partial_tv\in L^2(Q_{\mathcal T})$
imply $v\in C([0,\mathcal T]; H^1(\Omega))$,
from which, using the weak formula of weak solution to (\ref{paraboic2}),
one can see that $v(\cdot,0)=v_0$. Hence the initial condition is
fulfilled. Thanks to the regularities of $v$, by integration by parts
in the weak formula of (\ref{paraboic2}), one can further see that the
boundary conditions are satisfied in the sense of trace.
Therefore, $v$ is a strong solution to (\ref{paraboic2}), and satisfies the
desired estimate. This completes the
proof of Proposition \ref{prop.weak.str.par.hom}.
\end{proof}

Recalling (\ref{5.28.1}), as a direct corollary of Proposition \ref{prop.weak.str.par.hom}, one obtains
the existence and uniqueness of weak and strong solutions to (\ref{paraboic1}), which is stated in the following:

\begin{cor}\label{corexistparabolic}
Given a positive time $\mathcal T\in(0,\infty)$ and the initial data $u_0\in L^2(\Omega)$. We assume that (\ref{asum1}) and (\ref{asum2}) hold, and that
$f\in L^2(Q_{\mathcal T})$. Set
\beq
M_0=\|\varphi\|_{L^2(0,\mathcal T;H^{\frac12}(\Gamma_\ll))} +\|\psi\|_{L^2(0,\mathcal T;H^{\frac12}(\Gamma_{01}))}+\|\partial_t\Phi\|_{L^2(0,\mathcal T; L^2(\partial\Omega))}.
\eeq

Then, there is a unique weak solution to (\ref{paraboic1}), satisfying
\beq
\|u\|_{L^\infty(0,\mathcal T; L^2(\Omega))}+\|u\|_{L^2(0,\mathcal T; H^1(\Omega))}
\leq C(\|f\|_{L^2(Q_{\mathcal T})} +\|u_0\|_{L^2(\Omega)}+M_0),
\eeq
for a positive constant $C$ depending only on $\lambda$ and $\mathcal T$.

Moreover, if we assume in addition that $u_0\in H^1(\Omega)$, then the unique weak solution is a strong one, and satisfies
\begin{equation}
\|v\|_{L^\infty(0,\mathcal T; H^1(\Omega))}+\|v\|_{L^2(0,\mathcal T; H^2(\Omega))}+\|\partial_tv\|_{L^2(Q_\mathcal T)}\leq C(\|f\|_{L^2(Q_{\mathcal T})}+\|u_0\|_{H^1(\Omega)}+M_0),
\end{equation}
for a positive constant $C$ depending only on $\Omega$, $\lambda$,
$\mathcal T$, the modulus of continuity of $a$, $\|a\|_{L^\infty(\Omega)}$,
$\|\partial_za\|_{L^\infty(\Omega)}$, $\|\alpha\|_{W^{1,\infty}(\Omega)}$
and $\|\beta\|_{W^{1,\infty}(\Omega)}$.
\end{cor}

\noindent\textbf{Acknowledgments.} SH thanks to the Austrian Science Fund for the support via
the Hertha-Firnberg project T-764. RK acknowledges support by the Deutsche Forschungsgemeinschaft, Grant SFB 1114, KL 611/24.
The work of EST was supported in part by the ONR grant N00014-15-1-2333 and the NSF grants DMS-1109640 and DMS-1109645.

\end{document}